\documentclass[10pt,oneside]{article}
\usepackage{times,amsmath,amsbsy,amssymb,amscd,mathrsfs,graphicx}
\usepackage{amssymb}
\usepackage{amsmath,bm}
\usepackage{amsmath,amsthm}

\usepackage{xcolor,color}
\usepackage[font=small,labelfont=md,textfont=it]{caption}
\usepackage{floatrow,float}
\usepackage[titletoc, title]{appendix}
\usepackage[colorlinks,linkcolor=blue,citecolor=blue]{hyperref}
\usepackage{longtable}
\usepackage{pgfplots}
\usepackage{diagbox}
\usepackage{booktabs,makecell,multirow}
\usepackage[capitalise, nosort]{cleveref}
\usepackage{algorithm}
\usepackage{algorithmic}
\usepackage{amsmath,bm}
\usepackage{cases}
\usepackage{geometry}
\usepackage{extarrows}
\usepackage{tcolorbox}
\usepackage{bbm}
\usepackage{syntonly}
\usepackage[figuresright]{rotating}
 \usepackage{epstopdf}

\crefname{equation}{}{}
\crefname{lem}{Lemma}{Lemmas}
\crefname{thm}{Theorem}{Theorems}

\newcommand{\whk}[0]{\widehat{\lambda}_k}

\newcommand{\dd}{\,{\rm d}}
\newcommand{\R}{\,{\mathbb R}}

\newcommand{\dual}[1]{\left\langle {#1} \right\rangle}
\newcommand{\prox}[0]{ {\bf prox}}

\newcommand{\argmin}[0]{ {\mathop{{\rm  argmin}}\,}}

\newcommand{\st}[0]{ {{\rm  s.t.}\,}}

\newcommand{\nm}[1]{\left\lVert {#1} \right\rVert}
\newcommand{\snm}[1]{\left\lvert {#1} \right\rvert}

\newcommand{\ssnm}[1]
{
	\left\vert\kern-0.25ex
	\left\vert\kern-0.25ex
	\left\vert
	{#1}
	\right\vert\kern-0.25ex
	\right\vert\kern-0.25ex
	\right\vert
}

\makeatletter
\def\spher@harm#1{%
	\vbox{\hbox{%
			\offinterlineskip
			\valign{&\hb@xt@2\p@{\hss$##$\hss}\vskip.2ex\cr#1\crcr}%
		}\vskip-.36ex}%
}
\def\gshone{\spher@harm{.}}
\def\gshtwo{\spher@harm{.&.}}
\def\gshthree{\spher@harm{.&.&.}}
\let\gsh\spher@harm
\makeatother

\newtheorem{lem}{Lemma}[section]

\newtheorem{remark}{Remark}[section]

\newtheorem{thm}{Theorem}[section]

\newcolumntype{I}{!{\vrule width 1,5pt}}
\newlength\savedwidth

\newlength\savewidth

\newcounter{mnote}
\let\oldmarginpar\marginpar
\renewcommand\marginpar[1]
{\-\oldmarginpar[\raggedleft\footnotesize #1]{\raggedright\footnotesize #1}}

\makeatletter\def\@captype{table}\makeatother

\newfloatcommand{capbtabbox}{table}[][\FBwidth]
\usepackage[T1]{fontenc} 
\usepackage[utf8]{inputenc} 
\usepackage{authblk}
\begin{document}
\title{
  \Large \bf Accelerated primal-dual methods for linearly constrained convex optimization problems
}
\author{Hao Luo\thanks{School of Mathematical Sciences, Peking University, Beijing, 100871, China. Email: luohao@math.pku.edu.cn}}
\date{}
\maketitle

\begin{abstract}
		This work proposes an accelerated primal-dual dynamical system for affine constrained convex optimization and presents a class of primal-dual methods with nonergodic convergence rates. In continuous level, exponential decay of a novel Lyapunov function is established and in discrete level, implicit, semi-implicit and explicit numerical discretizations for the continuous model are considered sequentially and lead to new accelerated primal-dual methods for solving linearly constrained optimization problems. Special structures of the subproblems in those schemes are utilized to develop efficient inner solvers. In addition, nonergodic convergence rates in terms of  primal-dual gap, primal objective residual and feasibility violation are proved via a tailored discrete Lyapunov function. Moreover, our method has also been applied to decentralized distributed optimization for fast and efficient solution.
\end{abstract}


\section{Introduction}
In this paper, we are concerned with primal-dual methods for linearly constrained convex optimization:
\begin{equation}\label{eq:min-f-X-Ax-b}
	\min_{x\in \mathcal X}\,f(x)\quad \st~Ax = b,
\end{equation}
where $A\in\R^{m\times n},\,b\in\R^m,\,f:\R^n\to\R\cup\{+\infty\}$ is proper, closed and convex but possibly nonsmooth and $\mathcal X\subset \R^n$ is some (simple) closed convex set such as the box or the half space.
Through out, the domain of $f$
is assumed to have nonempty intersection with $\mathcal X$;
also, to promise nonempty feasible set, the vector $b$ shall belong
to the image of $\mathcal X$ under the linear transform $A:\R^n\to\R^m$.

The well-known augmented Lagrangian method (ALM) for \cref{eq:min-f-X-Ax-b} can be dated back to \cite{hestenes_multiplier_1969}. It recovers the proximal point algorithm for the dual problem of  \cref{eq:min-f-X-Ax-b} (cf. \cite{rockafellar_augmented_1976}) and is also equivalent to the Bregman method \cite{yin_bregman_2008} for total variation-based image restoration. 
Accelerated variants of the classical ALM using extrapolation technique \cite{Nesterov1983,tseng_on_accelerated_Seattle_2008} for the multiplier are summarized as follows. For smooth objective, He and Yuan \cite{he_acceleration_2010} proposed an accelerated ALM. Later in \cite{kang_accelerated_2013}, this was extended to nonsmooth case, and further generalizations such as inexact version and linearization can be found in \cite{huang_accelerated_2013,kang_inexact_2015}. For strongly convex but not necessarily smooth objective, Tao and Yuan \cite{tao_accelerated_2016} proposed an accelerated Uzawa method. We note that those accelerated methods mentioned here share the same nonergodic convergence rate $O(1/k^2)$ for the dual variable $\lambda_k$ (or the nonnegative residual $\mathcal L(x^*,\lambda^*)-\mathcal L(x_k,\lambda_k)$ which is (approximately) equal to the dual objective residual). 

To get nonergodic rates for the primal objective residual $\snm{f(x_k)-f(x^*)}$ and the feasibility violation $\nm{Ax_k-b}$, quadratic penalty  with continuation \cite{Lan2013} is sometimes combined 
with extrapolation. The accelerated quadratic penalty (AQP) method in \cite{li_convergence_2017} was proved to enjoy the rates $O(1/k)$ and $O(1/k^2)$, respectively for convex and strongly convex cases. In \cite{Xu2017}, a partially linearized accelerated proximal ALM was proposed and the sublinear rate $O(1/k^2)$ has been established for convex objective. However, for strongly convex case, the convergence rate of the fully linearized proximal ALM in \cite{Xu2017} is in ergodic sense. Based on Nesterov's smoothing technique \cite{nesterov_excessive_2005,nesterov_smooth_2005}, Tran-Dinh et al. \cite{patrascu_adaptive_2017,tran-dinh_constrained_2014,tran-dinh_primal-dual_2015,tran-dinh_smooth_2018} developed a primal-dual framework for linearly constrained convex optimization and applied it to \cref{eq:min-f-X-Ax-b} to obtain accelerated rates in nonergodic sense. Sabach and  Teboulle \cite{sabach_faster_2020} also presented a novel algorithm framework that can be used to \cref{eq:min-f-X-Ax-b} for nonergodic convergence rate.

For linear and quadratic programmings, superlinearly convergent semi-smooth Newton (SsN) based proximal augmented Lagrangian methods have been proposed in \cite{li_asymptotically_2020,niu_sparse_2021}.
It is worth noticing that Salim et al. \cite{salim_optimal_2021} developed a linearly convergent  primal-dual algorithm for problem \cref{eq:min-f-X-Ax-b} with strongly convex smooth objective and full column rank $A$. This method requires an inner Chebyshev iteration that plays the role of precondition and has been proved to achieve the complexity lower bound $\sqrt{\kappa_f\chi}\snm{\ln\epsilon}$, where $\kappa_f$ and $\chi$ are the condition numbers of $f$ and $A^{\top}A$, respectively.

On the other hand, some continuous-time primal-dual dynamical models for \cref{eq:min-f-X-Ax-b} have been developed as well. In \cite{zeng_distributed_2018}, Zeng et al. proposed two continuous models, and with strictly convex assumption, they proved the decay rate $O(1/t)$ for the primal-dual gap in ergodic sense. In \cite{zeng_dynamical_2019}, the asymptotic vanishing damping model \cite{su_dierential_2016} for unconstrained optimization 
was extended to a continuous-time primal-dual accelerated method with 
the decay rate $O(1/t^{2})$. We refer to \cite{bot_improved_2021,he_perturbed_2021} for more generalizations. 
However, none of the above works considered numerical discretizations for their models and developed new primal-dual algorithms. Recently, in \cite{he_convergence_2021,he_fast_2021,he_inertial_2021}, He et al. extended the inertial primal-dual dynamical system in \cite{zeng_distributed_2018} to obtain faster decay rates, by introducing suitable time scaling factors. They also proposed primal-dual methods based on proper time discretizations and proved nonergodic rate $O(1/k^2)$ for convex objective. In addition, for implicit scheme, linear rate has been proved by means of time rescaling effect.
For the two block case: 
\begin{equation}\label{eq:2b}
	f(x) =f_1(x_1)+ f_2(x_2),\quad Ax = A_1x_1+A_2x_2,
\end{equation}
more primal-dual dynamical systems can be found in \cite{attouch_fast_2021,franca_admm_2018,franca_nonsmooth_2021,he_convergence_2020}. In this setting, or even more general multi-block case (cf. \cref{eq:nb}), the alternating direction method of multiplies (ADMM) is one of the most prevailing splitting algorithms. We refer to \cite{boyd_distributed_2010,chambolle_first-order_2011,chambolle_ergodic_2016,chen_direct_2016,goldstein_fast_2014,han_linear_2015,He_Yuan2012,he_non-ergodic_2015,Li2019,lin_iteration_2015,liu_partial_2018,yuan_discerning_2020} and the references therein.

The remainder of this paper is organized as follows. In the rest of the introduction part, we continue with some essential notations and briefly summarize our main results. In \cref{sec:apd}, the accelerated primal-dual flow model is introduced and the exponential decay shall be established as well. Then, implicit, semi-implicit and explicit discretization  are considered sequentially from \cref{sec:apd-im,sec:apd-ex-x-im-l,sec:apd-correc-ex-x-im-l,sec:apd-ex-x-ex-l}, and nonergodic convergence rates are proved via a unified discrete Lyapunov function. After that, numerical reports for decentralized distributed optimization are presented in \cref{sec:ddo}, and finally, some concluding remarks are given in \cref{sec:apd-conclu}.
\subsection{Notations}
Let $\dual{\cdot,\cdot}$ be the usual $l^2$-inner product and set $\nm{\cdot} = \sqrt{\dual{\cdot,\cdot}}$. 
For a proper, closed and convex function $g:\R^{n}\to\R\cup\{+\infty\}$, we say $g\in\mathcal S_{\mu}^0(\mathcal X)$ if $\mu\geqslant 0$ and
\begin{equation}\label{eq:def-mu}
	g(y)\geqslant g(y)+\dual{p,y-x}+\frac{\mu}{2}\nm{y-x}^2
	\quad\text{for all }x,\,y\in\mathcal X,
\end{equation}
where $p\in\partial g(x)$. Let $\mathcal S_{\mu}^1(\mathcal X)$ be the set of all continuous differentiable functions in $\mathcal S_{\mu}^0(\mathcal X)$, and moreover, if $g\in\mathcal S_{\mu}^1(\mathcal X)$ has $L$-Lipschitz continuous gradient:
\[
\dual{\nabla g(x)-\nabla g(y),x-y}\leqslant L\nm{x-y}^2
\quad\text{for all }x,\,y\in\mathcal X,
\]
then we say $g\in \mathcal S_{\mu,L}^{1,1}(\mathcal X)$.
If $\mathcal X=\R^n$, then the underlying space $\mathcal X$ shall be dropped for simplicity, e.g., $\mathcal S_{\mu}^0(\R^n) = \mathcal S_{\mu}^0$.

For any $\beta\geqslant 0$, we set $g_\beta=g+\beta/2\nm{Ax-b}^2$ and for $\sigma>0$, let $\ell_\sigma(x): = 1/(2\sigma)\nm{Ax-b}^2$. It is evident that if $g\in\mathcal S_{\mu}^{0}(\mathcal X) $, then $g_\beta\in\mathcal S_{\mu_\beta}^0(\mathcal X)$, where  $\mu_\beta=\mu+\beta\sigma_{\min}^2(A)$ with $\sigma_{\min}(A)\geqslant 0$ being the smallest singular value of $A$. In addition, if $g\in\mathcal S_{\mu,L}^{1,1}(\mathcal X)$, then $g_\beta\in\mathcal S_{\mu_\beta,L_\beta}^{1,1}(\mathcal X)$, where  $L_\beta=L+\beta\nm{A}^2$. Moreover, for $\eta>0$, let $\prox_{ \eta g}^{\mathcal X}:\R^n\to\mathcal X$ be the proximal operator of $g$ over $\mathcal X$:
\begin{equation}\label{eq:prox-X}
	\prox^{\mathcal X}_{\eta g}(x): = \mathop{\rm argmin}\limits_{y\in\mathcal X}\left\{
	g(y)+\frac{1}{2\eta}\nm{y-x}^2
	\right\}
	\quad\text{for all }x\in\R^n.
\end{equation}
It is clear that $\prox^{\mathcal X}_{\eta g} = \prox_{\eta (g+\delta_{\mathcal X})}$, where $\delta_{\mathcal X}$ denotes the indicator function of $\mathcal X$, and if $\mathcal X=\R^n$, then \cref{eq:prox-X} agrees with the conventional proximal operator $\prox_{\eta g}$.

Given any $\beta\geqslant 0$, define the augmented Lagrangian of \cref{eq:min-f-X-Ax-b} by that
\[
\mathcal L_\beta(x,\lambda): =  f_\beta(x) + \delta_{\mathcal X}(x)
+\dual{\lambda,Ax-b}
\quad \forall\,(x,\lambda)\in
\R^n\times\R^m,
\]
and for $\beta=0$, we write $\mathcal L(x,\lambda)=\mathcal L_0(x,\lambda)$. Let $(x^*,\lambda^*)$ be a saddle point of $\mathcal L(x,\lambda)$, which means
\[
\min_{x\in\R^n} \mathcal L(x,\lambda^*)= \mathcal L(x^*,\lambda^*)=	\max_{\lambda\in\R^m}\mathcal L(x^*,\lambda) ,
\]
then $(x^*,\lambda^*)$ also 
satisfies the Karush--Kuhn--Tucker (KKT) system
\begin{equation}\label{eq:KKT}
	\left\{
	\begin{split}
		&{}0=Ax^*-b,\\		
		&{}0\in \partial f(x^*)+N_{\mathcal X}(x^*)+A^{\top}\lambda^*,
	\end{split}
	\right.
\end{equation}
where $\partial f(x^*)$ denotes the subdifferential of $f$ at $x^*$ and $N_{\mathcal X}(x^*)$ is the {\it norm cone} of $\mathcal X$ at $x^*$, which is defined as $	N_{\mathcal X}(x^*): = \left\{y\in\R^n:\dual{y,z-x^*}\leqslant 0
\,\,\text{ for all } z\in \mathcal X\right\}$. Throughout, we assume \cref{eq:min-f-X-Ax-b} admits at least one KKT point $(x^*,\lambda^*)$ satisfying \cref{eq:KKT}.
\subsection{Summary of main results}
\label{sec:apd-main-res}
In this work, for problem \cref{eq:min-f-X-Ax-b} with $f\in\mathcal S_\mu^0(\mathcal X),\,\mu\geqslant 0$, 
we propose the {\it accelerated primal-dual} (APD) flow system
\begin{subnumcases}{}
	\theta\lambda'  ={}\nabla_\lambda \mathcal L_\beta(v,\lambda),
	\label{eq:apd-sys-lambda-intro}\\
	x' = {}v-x,
	\label{eq:apd-sys-x-intro}\\
	\gamma v'  \in{}\mu_\beta(x-v)-\partial_x \mathcal L_\beta(x,\lambda),
	\label{eq:apd-sys-v-intro}
\end{subnumcases}
where $\partial_x\mathcal L_\beta(x,\lambda)=\partial f_\beta(x)
+N_{\mathcal X}(x)+A^{\top}\lambda$ and 
the above two scaling factors $\theta$ and $\gamma$ satisfy $	\theta' ={}-\theta$ and $
\gamma' ={}\mu_\beta-\gamma$, respectively.
We also introduce a novel Lyapunov function
\begin{equation}\label{eq:Et-intro}
	\mathcal E(t) = \mathcal L_\beta(x(t),\lambda^*) - \mathcal L_\beta(x^*,\lambda(t)) + \frac{\gamma(t)}{2}\nm{v(t)-x^*}^2+\frac{\theta(t)}{2}\nm{\lambda(t)-\lambda^*}^2,
\end{equation}
and prove the exponential decay $\mathcal E(t)=O(e^{-t})$ uniformly for $\mu_\beta\geqslant 0$, under the smooth case $f\in\mathcal S_\mu^1$. For general nonsmooth case, i.e., the differential inclusion \cref{eq:apd-sys-lambda-intro} itself, solution existence in proper sense together with the exponential decay is not considered in this paper. In addition, compared with our previous first-order primal-dual flow system \cite{luo_primal-dual_2021}, the current model \cref{eq:apd-sys-v-intro}, together with its time discretizations presented in this work, can be viewed as accelerated extensions.

Nevertheless, a family of accelerated primal-dual algorithms for \cref{eq:min-f-X-Ax-b} are presented systematically from numerical discretizations of our APD flow \cref{eq:apd-sys-lambda-intro} and 
analyzed via a unified Lyapunov function
\begin{equation}\label{eq:apd-Ek-intro}
	\mathcal E_k=
	\mathcal L_\beta(x_k,\lambda^*)-\mathcal L_\beta(x^*,\lambda_k)
	+\frac{\theta_k}{2}\nm{\lambda_k-\lambda^*}^2
	+\frac{\gamma_k}{2}\nm{v_k-x^*}^2,
\end{equation}
which is a discrete analogue to \cref{eq:Et-intro}. We shall prove the contraction property
\[
\mathcal E_{k+1}-\mathcal E_k\leqslant -\alpha_k\mathcal E_{k+1}
\quad \text{for all } k\in\mathbb N, 
\]
and then derive the {\it nonergodic} convergence estimate
\[
\mathcal L(x_k,\lambda^*)-\mathcal L(x^*,\lambda_k)+\snm{f(x_k)-f(x^*)}+\nm{Ax_k-b}\leqslant C\theta_k,
\]
where $\theta_k$ gives explicit decay rate for each method and $C>0$ is some constant. 

All these methods differ mainly from the treatment for the subproblem \cref{eq:xk1}, and we give a brief summary as below.
\begin{itemize}
	\item For convex objective $f$, if we use the augmented proximal subproblem
	\begin{equation}\label{eq:xk1}
		x_{k+1}=\mathop{\argmin}\limits_{x\in\mathcal X}\left\{f(x)+\frac{\sigma_k}{2}\nm{Ax-b}^2+\frac{\alpha_k}{2}\nm{x-\widehat{x}_k}^2\right\},
	\end{equation}
	then we have linear rate; see the implicit scheme \cref{eq:apd-im-x-im-l-v} and \cref{thm:conv-apd-im-x-im-l}.
	\item If one only linearizes $f$ (when it is smooth or has smooth component $h$ such that $f = h+g$) , then the rate is $O(L/k^2)$, where $L$ denotes the Lipschitz constant of $\nabla f$ (or $\nabla h$); see the semi-implicit discretization \cref{eq:apd-ex-x-im-l-x-h-g} and \cref{thm:apd-ex-x-im-l-h-g}.
	\item If one only linearizes the augmented term $\nm{Ax-b}^2$ in \cref{eq:xk1}, then the rate becomes $O(\nm{A}/k)$; see another semi-implicit scheme \cref{eq:apd-im-x-ex-l-v} and \cref{thm:conv-apd-im-x-ex-l}.	
	\item If both $f$ and the augmented term are linearized, then the final convergence rate is $O((\nm{A}+\sqrt{L})/k)$; see the explicit discretization \cref{eq:apd-ex-x-ex-l-x-h-g} and \cref{thm:conv-apd-ex-x-ex-l-h-g}.
\end{itemize}
We note that, for convex case $\mu=0$, all of our methods listed above are close to those existing algorithms in \cite{he_convergence_2021,he_fast_2021,he_inertial_2021,li_convergence_2017,sabach_faster_2020,tran-dinh_constrained_2014,tran-dinh_primal-dual_2015,tran-dinh_smooth_2018,Xu2017}, and they share the corresponding nonergodic rates.
However, for strongly convex case $\mu>0$, the above three linearized methods can achieve faster convergence rates:
$O((1-\sqrt{\mu/L})^k)$, $O(\nm{A}^2/k^2)$, and $O((\nm{A}^2+L)/k^2)$, respectively. Particularly, in \cite{sabach_faster_2020,tran-dinh_constrained_2014,tran-dinh_primal-dual_2015,tran-dinh_smooth_2018}, the rate $O(\nm{A}^2/k^2)$ has been achieved with strongly convex objective. 

Both of the two methods \cref{eq:apd-im-x-ex-l-v,eq:apd-ex-x-ex-l-x-h-g} only involve the proximal calculation of $f$ (or its nonsmooth part $g$).
As for the implicit scheme \cref{eq:apd-im-x-im-l-v} and the semi-implicit discretization \cref{eq:apd-ex-x-im-l-x-h-g}, following the spirit from \cite{li_asymptotically_2020,luo_primal-dual_2021,niu_sparse_2021},  we can transform the related subproblems into some nonlinear equations (or linear SPD systems) with respect to the dual variable, and then develop efficient inner solvers, such as the SsN method (or the preconditioned conjugate gradient (PCG) iteration), provided that there has some additional special structure such as sparsity.

In this work, we have not considered the two block case \cref{eq:2b}, for which ADMM-type methods are more practical. Taking this into account, the implicit scheme \cref{eq:apd-im-x-im-l-l} and the semi-implicit one \cref{eq:apd-ex-x-im-l-x-h-g} can not be applied directly to \cref{eq:2b}. However, as byproducts, both the semi-implicit discretization \cref{eq:apd-im-x-ex-l-v} and the explicit one \cref{eq:apd-ex-x-ex-l-x-h-g} are available for \cref{eq:2b} and lead to linearized parallel ADMM-type methods; see more discussions in \cref{rem:jdmm}.
\section{Accelerated Primal-Dual Flow}
\label{sec:apd}
As a combination of the Nesterov accelerated gradient flow \cite{luo_accelerated_2021,luo_differential_2021} and the primal-dual flow 
\cite{luo_primal-dual_2021}, our accelerated primal-dual flow  reads as 
\begin{subnumcases}{}
	\theta\lambda'  ={}\nabla_\lambda \mathcal L_\beta(v,\lambda),
	\label{eq:apd-sys-lambda}\\
	x' = {}v-x,
	\label{eq:apd-sys-x}\\
	\gamma v'  \in{}\mu_\beta(x-v)-\partial_x \mathcal L_\beta(x,\lambda),
	\label{eq:apd-sys-v}
\end{subnumcases}
where  $\partial_x\mathcal L_\beta(x,\lambda)=\partial f_\beta(x)+N_{\mathcal X}(x)+A^{\top}\lambda$, $\theta$ and $\gamma$ are two built-in scaling factors governed respectively by
\begin{equation}\label{eq:theta-gamma}
	\theta' =-\theta,\quad 
	\gamma' ={}\mu_\beta-\gamma,
\end{equation}
with $\mu_\beta=\mu+\beta\sigma_{\min}^2(A)$ and the initial condition $(\theta(0),\gamma(0))=(\theta_0,\gamma_0)>0$. It is not hard to calculate explicit solution of \cref{eq:theta-gamma}:
\[
\theta(t) = {}\theta_0e^{-t},\quad 
\gamma(t) = {}\mu_\beta+(\gamma_0-\mu_\beta)e^{-t}.
\]
Therefore, both $\theta$ and $\gamma$ are positive and approach to 0 and $\mu_\beta$ respectively with exponential rate. In addition, we have that $\gamma_{\min}:=\min\{\mu_\beta,\gamma_0\}\leqslant \gamma(t)\leqslant \gamma_{\max}:=\max\{\mu_\beta,\gamma_0\}$ for all $t\geqslant 0$.
However for algorithm designing, we shall keep the differential equation formulation \cref{eq:theta-gamma} and treat $\theta$ and $\beta$ as unknowns.

For simplicity, in this section, we restrict ourselves to the smooth case: $\mathcal X  = \R^n$ and $f\in\mathcal S_{\mu,L}^{1,1}$, for which unique classical solution to \cref{eq:apd-sys-lambda} can be obtained easily since now $\mathcal L_\beta(x,\lambda)$ is linear with respect to $\lambda$ and $L_\beta$-Lipschitz continuous in terms of $x$. The general nonsmooth case $f\in\mathcal S_\mu^0(\mathcal X)$, however,
deserves further investigation on the solution existence in proper sense, which together the nonsmooth version of \cref{lem:dE}, is beyond the scope of this work.

Now, our APD flow \cref{eq:apd-sys-lambda} becomes
\begin{subnumcases}{}
	\theta\lambda'  ={}\nabla_\lambda \mathcal L_\beta(v,\lambda),
	\label{eq:apd-sys-lambda-smooth}\\
	x' = {}v-x,
	\label{eq:apd-sys-x-smooth}\\
	\gamma v'  ={}\mu_\beta(x-v)-\nabla_x \mathcal L_\beta(x,\lambda),
	\label{eq:apd-sys-v-smooth}
\end{subnumcases}
with initial condition  $(\lambda(0),x(0),v(0))=(\lambda_0,x_0,v_0)\in\R^m\times \R^n\times \R^n$. Applying standard well-posedness theory of ordinary differential equations implies that the system \cref{eq:apd-sys-lambda-smooth} admits a unique solution $(\lambda,x,v)\in C^2([0,\infty);\R^m)\times C^2([0,\infty);\R^n)\times C^1([0,\infty);\R^n)$.

Let us equip the system \cref{eq:apd-sys-lambda-smooth} with a suitable Lyapunov function 
\begin{equation}\label{eq:Et}
	\mathcal E(t) := \mathcal L_\beta(x(t),\lambda^*) - \mathcal L_\beta(x^*,\lambda(t)) +\frac{\gamma(t)}{2}\nm{v(t)- x^*}^2
	+\frac{\theta(t)}{2}\nm{\lambda(t)-\lambda^*}^2,
\end{equation}
where $0\leqslant t<\infty$.  The following lemma establishes the exponential decay of \cref{eq:Et}, which holds uniformly for $\mu\geqslant 0$.
\begin{lem}\label{lem:dE}
	Assume $f\in\mathcal S_{\mu,L}^{1,1}$ with $0\leqslant \mu\leqslant L<\infty$ and let $(\lambda,x,v)$ be the unique solution to \cref{eq:apd-sys-x-smooth}, then for $\mathcal E(t)$ defined by \cref{eq:Et}, it holds that
	\begin{equation}\label{eq:dE}
		\frac{\dd}{\dd t}\mathcal E(t)
		\leqslant -\mathcal E(t)
		-\frac{\mu_\beta}{2}\nm{x'(t)}^2,
	\end{equation}
	which implies 
	\begin{equation}\label{eq:exp-Et}
		\mathcal E(t)+\frac{\mu_\beta}{2}\int_{0}^{t}e^{s-t}
		\nm{x'(s)}^2{\rm d}s\leqslant 
		e^{-t}\mathcal E(0),\quad 0\leqslant t<\infty.
	\end{equation}
	Moreover, $\nm{Ax(t)-b}\leqslant e^{-t}\mathcal R_0$ and $\snm{f(x(t))-f(x^*)}\leqslant e^{-t}\big(\mathcal E(0)+\mathcal R_0\nm{\lambda^*}\big)$, where $\mathcal R_0=\sqrt{2\theta_0\mathcal E(0)}+\theta_0\nm{\lambda_0-\lambda^*}+\nm{Ax_0-b}$.
\end{lem}
\begin{proof}
	Notice that $\mathcal L_\beta(x^*,\lambda) = f(x^*)$ is a constant for all $\lambda$. This fact will also be used implicitly somewhere else.
	A direct computation gives
	\[
	\begin{aligned}
		\frac{\dd}{\dd t}\mathcal E(t) ={}&
		\dual{x',\nabla_x\mathcal L_\beta(x,\lambda^*)}
		+ \frac{\gamma'}{2}\nm{v-x^*}^2+\dual{\gamma v',v-x^*}
		+	\frac{\theta'}{2}\nm{\lambda-\lambda^*}^2	+\dual{\theta\lambda',\lambda-\lambda^*}.
	\end{aligned}
	\]
	In view of \cref{eq:theta-gamma,eq:apd-sys-x-smooth}, we replace all the derivatives with their right hand sides and obtain $\mathcal E'(t)=I_1+I_2$, where
	\[
	\begin{aligned}
		I_1:={}&	-\frac{\theta}{2}\nm{\lambda-\lambda^*}^2
		+ \frac{\mu_\beta-\gamma}{2}\nm{v-x^*}^2
		+\mu_\beta\dual{x-v,v-x^*},\\
		I_2:={}&\dual{\nabla_x\mathcal L_\beta(x,\lambda^*),v-x}-\dual{\nabla_x\mathcal L_\beta(x,\lambda),v-x^*}
		+\dual{\nabla_\lambda\mathcal L_\beta(v,\lambda),\lambda-\lambda^*}.
	\end{aligned}
	\]
	Recall the identity 
	\begin{equation}\label{eq:3id}
		\mu_\beta\dual{x-v,v-x^*} = \frac{\mu_\beta}{2}
		\left(
		\nm{x-x^*}^2-\nm{v-x^*}^2-\nm{v-x}^2
		\right),
	\end{equation}
	which is trivial but very useful in our later analysis. We rewrite $I_1$ as follows
	\begin{equation}\label{eq:eq-I1}
		I_1={}\frac{\mu_\beta}{2}\nm{x-x^*}^2-\frac{\gamma}{2}\nm{v-x^*}^2
		-\frac{\theta}{2}\nm{\lambda-\lambda^*}^2-\frac{\mu_\beta}{2}\nm{v-x}^2.
	\end{equation}
	Inserting the splitting 
	\[
	\dual{\nabla_x\mathcal L_\beta(x,\lambda^*),v-x} = 
	\dual{\nabla_x\mathcal L_\beta(x,\lambda^*),x^*-x}
	+	\dual{\nabla_x\mathcal L_\beta(x,\lambda^*),v-x^*}
	\]
	into $I_2$ and using $\nabla_x\mathcal L_\beta(x,\lambda^*)-\nabla_x\mathcal L_\beta(x,\lambda)
	=A^{\top}(\lambda^*-\lambda)$, we find
	\[
	I_2 = \dual{\nabla_x\mathcal L_\beta(x,\lambda^*),x^*-x}
	+\dual{A^{\top}(\lambda^*-\lambda),v-x^*}
	+\dual{\nabla_\lambda\mathcal L_\beta(v,\lambda),\lambda-\lambda^*}.
	\]
	Thanks to \eqref{eq:apd-sys-lambda-smooth} and the optimality condition \cref{eq:KKT}, i.e., $Ax^* = b$, the 
	sum of last two terms vanishes. Hence, it follows from the fact $\mathcal L_\beta(\cdot,\lambda^*)\in\mathcal S_{\mu_\beta}^1$ that (cf. \cref{eq:def-mu})
	\[
	\begin{aligned}
		I_2	 	\leqslant {}&\mathcal L_\beta(x^*,\lambda^*)	-\mathcal L_\beta(x,\lambda^*)	
		-\frac{\mu_\beta}{2}\nm{x-x^*}^2
		={}\mathcal L_\beta(x^*,\lambda)	-\mathcal L_\beta(x,\lambda^*)	
		-\frac{\mu_\beta}{2}\nm{x-x^*}^2.	
	\end{aligned}
	\]
	Now, in view of $x'=v-x$, collecting the above estimate and \cref{eq:eq-I1} implies \cref{eq:dE}. 
	
	From \cref{eq:dE} follows \cref{eq:exp-Et}, and analogous to \cite[Corollary 2.1]{luo_primal-dual_2021}, it is not hard to establish the exponential decay estimates of the feasibility violation $\nm{Ax(t)-b}$ and the primal objective residual $\snm{f(x(t))-f(x^*)}$. Consequently, this completes the proof of this lemma.
\end{proof}

\section{The Implicit Discretization}
\label{sec:apd-im}
From now on, we arrive at the discrete level and will consider several 
numerical discretizations for the APD flow system \cref{eq:apd-sys-lambda}. Those differential equation solvers mainly include an implicit Euler scheme \cref{eq:apd-im-x-im-l-v}, two semi-implicit schemes (cf. \cref{eq:apd-im-x-ex-l-v,eq:apd-ex-x-im-l-x-h-g}) and an explicit scheme \cref{eq:apd-ex-x-ex-l-x-h-g}, and are transformed into primal-dual algorithms for the original affine constrained convex optimization problem \cref{eq:min-f-X-Ax-b}. Nonergodic convergence rates will also be established via a unified discrete Lyapunov function.

In this section, let us start with the fully implicit Euler method:
\begin{subnumcases}{}
	\theta_k \frac{\lambda_{k+1}-\lambda_k}{\alpha_k} = {} \nabla_\lambda \mathcal L_\beta(v_{k+1},\lambda_{k+1}),\label{eq:apd-im-x-im-l-l}\\
	\frac{x_{k+1}-x_{k}}{\alpha_k}={}v_{k+1}-x_{k+1},\label{eq:apd-im-x-im-l-x}\\
	\gamma_k \frac{v_{k+1}-v_k}{\alpha_k} \in{}\mu_\beta(x_{k+1}-v_{k+1}) -\partial_x\mathcal L_\beta\left(x_{k+1}, \lambda_{k+1}\right),
	\label{eq:apd-im-x-im-l-v}
\end{subnumcases}
with initial guess $(x_0,v_0)\in\mathcal X\times \R^n$.
The scaling parameter system \cref{eq:theta-gamma} is discretized implicitly as follows
\begin{equation}\label{eq:apd-im-gamma-theta}
	\frac{\theta_{k+1}-\theta_{k}}{\alpha_k}=-\theta_{k+1},\quad 	
	\frac{\gamma_{k+1}-\gamma_{k}}{\alpha_k}={}\mu_\beta-\gamma_{k+1},
\end{equation}
with $\theta_0=1$ and $\gamma_0>0$. This will be used in all the forthcoming methods.

Before the convergence analysis, let us have a look at the solvability. 
By \eqref{eq:apd-im-x-im-l-x}, express $v_{k+1}$ in terms of $x_{k+1}$ and $x_k$ and plug it into  \eqref{eq:apd-im-x-im-l-l} and \eqref{eq:apd-im-x-im-l-v} to obtain
\begin{subnumcases}{}
	\lambda_{k+1}={}\lambda_k-\frac{1}{\theta_{k}}(Ax_k-b)+\frac{1}{\theta_{k+1}}(Ax_{k+1}-b),
	\label{eq:apd-im-x-im-l-lk1}\\
	x_{k+1}
	\in{}y_k
	-\eta_k\left(\partial f_\beta(x_{k+1})+N_{\mathcal X}(x_{k+1})+A^{\top}\lambda_{k+1}\right),
	\label{eq:apd-im-x-im-l-xk1}
\end{subnumcases}
where $\eta_k=\alpha_k^2/\tau_k$ and
\begin{equation}\label{eq:im-x-im-l-tauk-yk}
	\tau_k:={}\gamma_k+\mu_\beta\alpha_k+\gamma_k\alpha_k,\quad
	y_k:={} \tau_k^{-1}
	\left((\gamma_k+\mu_\beta\alpha_k)x_k+\gamma_k\alpha_kv_k\right).
\end{equation}
Eliminating $\lambda_{k+1}$ from \cref{eq:apd-im-x-im-l-xk1} to get
\begin{equation}
	\label{eq:apd-im-x-im-l-xk1-argmin}
	x_{k+1}
	=\mathop{ \argmin}\limits_{x\in\mathcal X}
	\left\{
	f_\beta(x)+\frac{1}{2\theta_{k+1}}\nm{Ax-b}^2
	+\frac{1}{2\eta_k}\nm{x-w_k}^2
	\right\},
\end{equation}
where $w_k :=y_k
-\eta_kA^{\top}\left(
\lambda_k-\theta_{k}^{-1}(Ax_k-b)\right)$. We note that except the augmented term in $f_\beta$, the quadratic penalty term $\ell_{\theta_{k+1}}(x)= 1/(2\theta_{k+1})\nm{Ax-b}^2 $ in \cref{eq:apd-im-x-im-l-xk1-argmin} comes from the implicit choice $\lambda_{k+1}$ in \eqref{eq:apd-im-x-im-l-v}, since it is coupled with $x_{k+1}$. If we drop that penalty term, then \cref{eq:apd-im-x-im-l-xk1-argmin} is very close to the classical proximal ALM. Clearly, we have $\{x_k\}\subset \mathcal X$ and once we get $x_{k+1}\in\mathcal X$ from \cref{eq:apd-im-x-im-l-xk1-argmin}, both $v_{k+1}$ and $\lambda_{k+1}$ are obtained sequentially. 

In addition, if $\beta=0$, then we may utilize the hidden structure of \cref{eq:apd-im-x-im-l-lk1} to solve it more efficiently. Indeed, by \eqref{eq:apd-im-x-im-l-xk1}, it follows that $x_{k+1} = \prox^{\mathcal X}_{\eta_kf}(y_k-\eta_k A^{\top}\lambda_{k+1})$,
which together with \eqref{eq:apd-im-x-im-l-lk1} gives
\begin{equation}
	\label{eq:apd-im-x-im-l-lk1-nonlinear}
	\theta_{k+1}\lambda_{k+1}-A\prox^{\mathcal X}_{\eta_kf}(y_k-\eta_k A^{\top}\lambda_{k+1})={}\theta_{k+1}
	\left(\lambda_k-\theta_{k}^{-1}(Ax_k-b)\right)-b.
\end{equation}
According to \cref{sec:SsN}, such a nonlinear equation may be solved via the SsN method (\cref{algo:SsN}). We stop the discussion here and put some remarks at the end of this section.

For convergence analysis, we introduce a tailored Lyapunov function
\begin{equation}\label{eq:Ek}
	\mathcal E_k:=
	\mathcal L_\beta(x_k,\lambda^*)-\mathcal L_\beta(x^*,\lambda_k)
	+\frac{\gamma_k}{2}\nm{v_k-x^*}^2
	+\frac{\theta_k}{2}\nm{\lambda_k-\lambda^*}^2,\quad k\in\mathbb N,
\end{equation}
which matches the discrete version of \cref{eq:Et}. 
\begin{thm}\label{thm:conv-apd-im-x-im-l}
	Assume $f\in\mathcal S_{\mu}^0(\mathcal X)$ with $\mu\geqslant 0$. 
	Then for the fully implicit scheme \cref{eq:apd-im-x-im-l-l} with $(x_0,v_0)\in\mathcal X\times \R^n$ and any $\alpha_k>0$, we have $\{x_k\}\subset \mathcal X$ and 
	\begin{equation}\label{eq:diff-Ek-im-x-im-l}
		\mathcal E_{k+1}-	\mathcal E_{k}
		\leqslant -\alpha_k		\mathcal E_{k+1}, \quad\text{for all~}k\in\mathbb N.
	\end{equation}
	Moreover, there holds that
	\begin{subnumcases}
		{}\nm{Ax_{k}-b}\leqslant \theta_k\mathcal R_0,\label{eq:conv-Axk-b-im-x-im-l}\\
		0\leqslant \mathcal L(x_{k},\lambda^*)-	\mathcal L(x^*,\lambda_{k})
		\leqslant \theta_k\mathcal E_0,
		\label{eq:conv-Lk-im-x-im-l}
		\\	
		{}\snm{f(x_k)-f(x^*)}
		\leqslant  \theta_k\left(\mathcal E_0+\mathcal R_0\nm{\lambda^*}\right),
		\label{eq:conv-fk-im-x-im-l}
	\end{subnumcases}
	where $\theta_k = \prod_{i=0}^{k-1}\frac{1}{1+\alpha_i}$ and 
	\begin{equation}\label{eq:R0}
		\mathcal R_0:=
		\sqrt{2\mathcal E_0}+
		\nm{\lambda_0-\lambda^*}+\nm{Ax_0-b}
	\end{equation}
\end{thm}
\begin{proof}
	Mimicking the proof of \cref{lem:dE}, we replace the derivative with the difference 
	$	\mathcal E_{k+1}-\mathcal E_{k} = I_1+I_2+I_3$, where
	\begin{equation}\label{eq:I1-I2-I3}
		\left\{
		\begin{aligned}
			I_1	:={}&\mathcal L_\beta\left(x_{k+1},\lambda^*\right)
			-\mathcal L_\beta\left(x_{k},\lambda^*\right),\\
			I_2:={}&\frac{\theta_{k+1}}{2}
			\nm{\lambda_{k+1}-\lambda^*}^2 - 
			\frac{\theta_{k}}{2}
			\nm{\lambda_{k}-\lambda^*}^2,\\
			I_3:=	{}&\frac{\gamma_{k+1}}{2}
			\nm{v_{k+1}-x^*}^2 - 
			\frac{\gamma_{k}}{2}
			\nm{v_{k}-x^*}^2.
		\end{aligned}
		\right.
	\end{equation}
	
	Let us set the first term $I_1$ aside and consider the estimates for $I_2$ and $I_3$. 
	For a start, 
	by the equation of $\{\theta_k\}$ in \eqref{eq:apd-im-gamma-theta}, an evident calculation yields that
	\begin{equation}
		\label{eq:I1-im-x-im-l}
		\begin{aligned}
			I_2=&\frac{\theta_{k+1}-\theta_{k}}{2}\nm{\lambda_{k+1}-\lambda^*}^2
			+\frac{\theta_{k}}{2}
			\left(\nm{\lambda_{k+1}-\lambda^*}^2 - 
			\nm{\lambda_{k}-\lambda^*}^2\right)\\
			=&	-\frac{\alpha_k\theta_{k+1}}{2}\nm{\lambda_{k+1}-\lambda^*}^2
			-\frac{\theta_{k}}{2}\nm{\lambda_{k+1}-\lambda_k}^2
			+\theta_{k}\dual{\lambda_{k+1}-\lambda_k,\lambda_{k+1}-\lambda^*}.
		\end{aligned}
	\end{equation}
	According to \eqref{eq:apd-im-x-im-l-l}, we rewrite the last cross term in \cref{eq:I1-im-x-im-l} and obtain
	\begin{equation}
		\label{eq:I1-im-x-im-l-simple}
		\begin{aligned}
			I_2=&-\frac{\alpha_k\theta_{k+1}}{2}\nm{\lambda_{k+1}-\lambda^*}^2
			-\frac{\theta_{k}}{2}\nm{\lambda_{k+1}-\lambda_{k}}^2
			+\alpha_{k}\dual{Av_{k+1}-b,\lambda_{k+1}-\lambda^*}.
		\end{aligned}
	\end{equation}
	Similarly, by \eqref{eq:apd-im-gamma-theta}, the term $I_3$ admits the decomposition
	\begin{equation}\label{eq:I2-1-im-x-ex-l}
		\begin{aligned}
			I_3={}&
			\frac{\alpha_k(\mu_\beta-\gamma_{k+1})}{2}
			\nm{v_{k+1}-x^*}^2-	\frac{\gamma_{k}}2\nm{v_{k+1}-v_k}^2
			+\gamma_{k}
			\dual{v_{k+1}-v_k, v_{k+1} -x^*}.
		\end{aligned}
	\end{equation}
	In view of \eqref{eq:apd-im-x-im-l-v}, it is not hard to find
	\[
	\gamma_{k}(v_{k+1}-v_k) = \mu_\beta\alpha_k(x_{k+1} - v_{k+1}) - \alpha_k\left(\xi_{k+1}+A^{\top}\lambda_{k+1}\right),
	\]
	where $\xi_{k+1}\in\partial f_\beta(x_{k+1})+N_{\mathcal X}(x_{k+1})$. Hence, $I_3$ can be further  expanded by that
	
	\begin{equation}\label{eq:I2-2-im-x-ex-l}
		\begin{aligned}
			I_3
			= {}&\mu_\beta\alpha_k\dual{x_{k+1} - v_{k+1}, v_{k+1} - x^{*}}
			-
			\alpha_k \dual{ \xi_{k+1}+A^{\top}\lambda^*, v_{k+1} - x^{*}}\\
			{}&\quad+	\frac{\alpha_k(\mu_\beta-\gamma_{k+1})}{2}
			\nm{v_{k+1}-x^*}^2		-	\frac{\gamma_{k}}2\nm{v_{k+1}-v_k}^2
			-\alpha_k\dual{Av_{k+1}-b,\lambda_{k+1}-\lambda^*},
		\end{aligned}
	\end{equation}
	where the last term in the above equality offsets the last term in \cref{eq:I1-im-x-im-l-simple}. By \cref{eq:3id}, the first cross term in \cref{eq:I2-2-im-x-ex-l} is rewritten as follows
	\begin{equation}\label{eq:cross}
		\begin{aligned}
			{}&			2\dual{x_{k+1} - v_{k+1}, v_{k+1} - x^{*}}=
			\left\|x_{k+1}-x^{*}\right\|^{2}
				-\|x_{k+1}-v_{k+1}\|^{2}
				-\left\|v_{k+1}-x^{*}\right\|^{2}.
		\end{aligned}
	\end{equation}
	Observing  \eqref{eq:apd-im-x-im-l-x}, we split the second cross term in \cref{eq:I2-2-im-x-ex-l} and get
	\[
	\begin{aligned}
		{}&
		-\alpha_k \dual{\xi_{k+1}+A^{\top}\lambda^*, v_{k+1} - x^{*}} 	
		={}
		-	\dual{ \xi_{k+1}+A^{\top}\lambda^*,x_{k+1}-x_k} 	-\alpha_k\dual{ \xi_{k+1}+A^{\top}\lambda^*, x_{k+1}-x^*},
	\end{aligned}
	\]
	By the fact that $\mathcal L_\beta(\cdot,\lambda^*)\in\mathcal S_{\mu_\beta}^{0}(\mathcal X)$ and $\xi_{k+1}+A^{\top}\lambda^*\in\partial_x\mathcal L_\beta(x_{k+1},\lambda^*)$, we obtain
	\begin{equation}\label{eq:est-mid-I3}
		\begin{aligned}
			-\alpha_k \dual{\xi_{k+1}+A^{\top}\lambda^*, v_{k+1} - x^{*}} 
			\leqslant {}&\mathcal L_\beta(x_k,\lambda^*)-\mathcal L_\beta(x_{k+1},\lambda^*)
			-\frac{\mu_\beta\alpha_k}{2}\nm{x_{k+1}-x^*}^2\\
			{}&\quad+\alpha_k\left(\mathcal L_\beta(x^*,\lambda^*)-\mathcal L_\beta(x_{k+1},\lambda^*)\right).
		\end{aligned}
	\end{equation}
	Note that the first term in \cref{eq:est-mid-I3} nullifies $I_1$ exactly.
	We find, after rearranging terms and dropping the surplus negative square term $-\nm{x_{k+1}-v_{k+1}}^2$, that
	\begin{equation}\label{eq:I1+I2+I3-im-x-im-l}
		\begin{aligned}
			\mathcal E_{k+1}-	\mathcal E_{k}\leqslant &-\alpha_{k}	\mathcal E_{k+1}
			-\frac{\theta_{k}}{2}\nm{\lambda_{k+1}-\lambda_{k}}^2
			-	\frac{\gamma_{k}}2\nm{v_{k+1}-v_k}^2,
		\end{aligned}
	\end{equation}
	which implies \cref{eq:diff-Ek-im-x-im-l} immediately.
	
	By the equation of $\{\theta_k\}$ in \cref{eq:apd-im-gamma-theta}, we have $\theta_k = \prod_{i=0}^{k-1}\frac{1}{1+\alpha_i}$, and from \eqref{eq:diff-Ek-im-x-im-l} follows  $	\mathcal E_k
	\leqslant \theta_k\mathcal E_0$, which promises \eqref{eq:conv-Lk-im-x-im-l}. So it is enough to establish
	\cref{eq:conv-Axk-b-im-x-im-l}.
	By \eqref{eq:apd-im-x-im-l-l}, we find
	\begin{equation}\label{eq:lk1-lk}
		\lambda_{k+1} 
		={}\lambda_k -\frac{1}{\theta_{k}}(Ax_k-b)+\frac{1}{\theta_{k+1}}(Ax_{k+1}-b).
	\end{equation}
	Whence, it follows that
	\begin{equation}\label{eq:lk-l0}
		\lambda_k-\frac1{\theta_k}(Ax_k-b) 
		= \lambda_0-(Ax_0-b),
		\quad k\in\mathbb N,
	\end{equation}
	which implies the inequality
	\[
	\begin{aligned}\nm{Ax_k-b}={}&\theta_k\nm{\lambda_k-\lambda_0
			+(Ax_0-b)}
		\leqslant{}
		\theta_k\nm{\lambda_k-\lambda_0}
		+\theta_k\nm{Ax_0-b}.
	\end{aligned}
	\]
	Thanks to the estimate $	\mathcal E_k
	\leqslant \theta_k\mathcal E_0$, we have $\nm{\lambda_k-\lambda^*}^2
	\leqslant 2\mathcal E_0$ and moreover,
	\[
	\begin{aligned}\nm{Ax_k-b}
		\leqslant{}& \theta_k\nm{\lambda_k-\lambda^*}+\theta_k\nm{\lambda_0-\lambda^*}
		+\theta_k\nm{Ax_0-b}
		\leqslant \theta_k\mathcal R_0,
	\end{aligned}
	\]
	which proves \eqref{eq:conv-Axk-b-im-x-im-l}. In addition, it is clear that
	\[
	\begin{aligned}
		0\leqslant 	\mathcal L(x_k,\lambda^*)-\mathcal L(x^*,\lambda_k)
		={}&f(x_k)-f(x^*)+\dual{\lambda^*,Ax_k-b}
		\leqslant{}
		\mathcal L_\beta(x_k,\lambda^*)-
		\mathcal L_\beta(x^*,\lambda_k)\leqslant \theta_k
		\mathcal{E}_0,
	\end{aligned}
	\]
	and thus there holds
	\[
	\begin{aligned}
		\snm{f(x_k)-f(x^*)}\leqslant	{}& 
		\snm{\dual{\lambda^*,Ax_k-b}}+
		\theta_k
		\mathcal{E}_0
		\leqslant 	{}\theta_k\left(
		\mathcal{E}_0+\nm{\lambda^*}\mathcal R_0
		\right).
	\end{aligned}
	\]
	This establishes \eqref{eq:conv-fk-im-x-im-l} and finishes the proof of this theorem.
\end{proof}

To the end, let us make some final remarks on the implicit discretization \cref{eq:apd-im-x-im-l-l}.
First of all, the augmented term $\beta/2\nm{Ax-b}^2$ in $f_\beta$ is different from the penalty term $\ell_{\theta_{k+1}}(x)=1/(2\theta_{k+1})\nm{Ax-b}^2$ in \cref{eq:apd-im-x-im-l-xk1-argmin}. The latter is mainly due to the implicit discretization of $\lambda$ in $\partial_x\mathcal L_\beta(x,\lambda)$, which is coupled with $v$ and therefore $x$, by \eqref{eq:apd-im-x-im-l-l} and \eqref{eq:apd-im-x-im-l-x}. The former makes sense only in the case that $\sigma_{\min}(A)>0$, which brings strong convexity to $f_\beta$ and promises $\mu_\beta=\mu+\beta\sigma_{\min}^2(A)>0$ even if $f$ is only convex (i.e., $\mu=0$). However, $\sigma_{\min}(A)>0$ means $A$ has full column rank. We are not assuming that this must be true throughout the paper but just want to be benefit from this situation. On the other hand, \cref{thm:conv-apd-im-x-im-l} implies the convergence rate has nothing to do with $\mu $ and $\mu_\beta$. Hence, for the implicit Euler method \cref{eq:apd-im-x-im-l-v}, there is no need to call these two parameters. Below, we summarize \cref{eq:apd-im-x-im-l-v} in \cref{algo:Im-APD} by setting $\mu=0$ and $\beta=0$. 
\begin{algorithm}[H]
	\caption{Implicit APD method for \cref{eq:min-f-X-Ax-b} with $f\in\mathcal S_{0}^{0}(\mathcal X)$.}
	\label{algo:Im-APD}
	\begin{algorithmic}[1] 
		\REQUIRE  $\theta_0=1,\,\gamma_0>0,\, (x_0,v_0)\in\mathcal X\times \R^n,\,\lambda_0\in\R^m$.
		\FOR{$k=0,1,\ldots$}
		\STATE Choose step size $\alpha_k>0$.
		\STATE Update $\displaystyle \theta_{k+1}= \theta_k/(1+\alpha_k)$ and compute $\gamma_k=\theta_k\gamma_0$.
		\STATE Solve $(\lambda_{k+1},x_{k+1})$ from \cref{eq:apd-im-x-im-l-lk1} with $\beta=0$ and $\mu_\beta=0$. This reduces to either \cref{eq:apd-im-x-im-l-xk1-argmin} or \cref{eq:apd-im-x-im-l-lk1-nonlinear}.
		\STATE Update $\displaystyle v_{k+1} =x_{k+1}+(x_{k+1}-x_k)/\alpha_k$. 
		\ENDFOR
	\end{algorithmic}
\end{algorithm}

Secondly, it is not surprising to see the unconditional contraction \cref{eq:diff-Ek-im-x-im-l}, which corresponds to the continuous case \cref{eq:dE}. 
In other words, fully implicit scheme is more likely to inherit core properties, such as exponential decay and time scaling, from the continuous level. Indeed, the exponential decay $O(e^{-t})$ in \cref{eq:exp-Et} is nothing but the time scaling effect, and it has been maintained by \cref{eq:apd-im-x-im-l-x} since we have no restriction on the step size $\alpha_k$. This can also be observed from \cite{he_convergence_2021,he_fast_2021,luo_primal-dual_2021}, and even for unconstrained problems  \cite{attouch_fast_2019,chen_first_2019,luo_differential_2021}. 
If $\alpha_k\geqslant \alpha_{\min}>0$, then the linear rate $(1+\alpha_{\min})^{-k}$ follows, and if we choose $\tau_k = \alpha_k^2$, then by \cref{eq:est-1}, we have the sublinear rate $O(1/k^2)$.

Thirdly, one may observe the relation \cref{eq:lk-l0}, which allows us to drop the sequence $\{\lambda_k\}$ and simplify \cref{algo:Im-APD}. This particular feature exists in all the forthcoming algorithms, and thus they can be simplified as possible as we can. But dropping $\{\lambda_k\}$ means we shall solve $x_{k+1}$ from the inner problem \cref{eq:apd-im-x-im-l-xk1-argmin}, which calls the proximal calculation of $f_{\beta}+\ell_{\theta_{k+1}}$ over $\mathcal X$. In some cases, it would be better to keep $\{\lambda_k\}$ as it is and consider the inner problem with $\lambda_{k+1}$, as discussed before on \cref{eq:apd-im-x-im-l-lk1-nonlinear}, which can be solved via the SsN method if $\prox_{\eta f}^{\mathcal X}$ is semi-smooth and has special structure. 
However, no matter which subproblem, proximal calculation of $f_{\beta}$ or $f$ may not be easy, especially for the composite case $f = h+g$. 

Finally, the implicit scheme \cref{eq:apd-im-x-im-l-l}, as well as the semi-implicit one \cref{eq:apd-ex-x-im-l-x-h-g}, can not lead to ADMM-type methods when applied to the two block case \cref{eq:2b}, since the augmented term still exists (even for $\beta=0$) and it makes $x_1$ and $x_2$ coupled with each other. However, for \cref{eq:apd-im-x-ex-l-v} and \cref{eq:apd-ex-x-ex-l-x-h-g}, they lead to linearized parallel ADMM-type methods; see \cref{rem:jdmm}.

Nevertheless, we shall emphasis that, the implicit scheme \cref{eq:apd-im-x-im-l-x} renders us some useful aspects. Nonergodic convergence rates analysis of all the forthcoming algorithms are followed from it and based on the unified Lyapunov function \cref{eq:Ek}. Also, it motivates us to consider semi-implicit and explicit discretizations, which bring linearization and lead to better primal-dual algorithms.

\section{A Semi-implicit Discretization}
\label{sec:apd-ex-x-im-l}
As we see, the implicit choice $\lambda_{k+1}=\lambda_k+\alpha_k/\theta_k(Av_{k+1}-b)$ in \eqref{eq:apd-im-x-im-l-v} makes $x_{k+1}$ and $\lambda_{k+1}$ coupled with each other. It is natural to consider the explicit one
\begin{equation}\label{eq:barl-k}
	\whk=
	\lambda_k+\frac{\alpha_k}{\theta_k}\left(
	Av_k-b
	\right),
\end{equation}
which gives a semi-implicit discretization
\begin{subnumcases}{}
	\theta_k \frac{\lambda_{k+1}-\lambda_k}{\alpha_k} = {} \nabla_\lambda \mathcal L_\beta(v_{k+1},\lambda_{k+1}),\label{eq:apd-im-x-ex-l-l}\\
	\frac{x_{k+1}-x_{k}}{\alpha_k}={}v_{k+1}-x_{k+1},\label{eq:apd-im-x-ex-l-x}\\
	\gamma_k \frac{v_{k+1}-v_k}{\alpha_k} \in{}\mu_\beta(x_{k+1}-v_{k+1}) -\partial_x\mathcal L_\beta(x_{k+1}, \whk).
	\label{eq:apd-im-x-ex-l-v}
\end{subnumcases}

Being different from $\lambda_{k+1} $, the explicit choice \cref{eq:barl-k} brings the gap $A(v_{k+1}-v_k)$, which can be controlled by  the additional negative term $-\nm{v_{k+1}-v_k}^2$ in \cref{eq:I1+I2+I3-im-x-im-l}.
Again, the initial guess is given by $(x_0,v_0)\in\mathcal X\times \R^n$, and the parameter system \cref{eq:theta-gamma} is still 
discretized by \cref{eq:apd-im-gamma-theta}.

Recall that $\gamma_{\min} = \min\{\mu_\beta,\gamma_{0}\}$ and $\gamma_{\max} = \max\{\mu_\beta,\gamma_{0}\}$. Moreover, from \cref{eq:apd-im-gamma-theta}, it is not hard to conclude that 
\[
\gamma_{\min}\leqslant \gamma_k\leqslant \gamma_{\max}, \quad\text{for all~}k\in\mathbb N.
\]
Let us first establish the contraction property of the Lyapunov function \cref{eq:Ek}, from which we 
can obtain nonergodic convergence rate as well. After that we discuss the solvability of \cref{eq:apd-im-x-ex-l-l} and summarize it in \cref{algo:Semi-APD}.
\begin{thm}\label{thm:conv-apd-im-x-ex-l}
	Assume $f\in\mathcal S_{\mu}^0(\mathcal X)$ with $\mu\geqslant 0$. 
	Then for the semi-implicit scheme \cref{eq:apd-im-x-ex-l-l} with initial guess $(x_0,v_0)\in\mathcal X\times \R^n$ and the relation $\gamma_{k}\theta_k=\nm{A}^2\alpha_k^2$, we have $\{x_k\}\subset \mathcal X$ and 
	\begin{equation}\label{eq:diff-Ek-im-x-ex-l}
		\mathcal E_{k+1}-	\mathcal E_{k}
		\leqslant -\alpha_k		\mathcal E_{k+1}, \quad\text{for all~}k\in\mathbb N.
	\end{equation}
	Moreover, it holds that
	\begin{subnumcases}
		{}\nm{Ax_{k}-b}\leqslant \theta_k\mathcal R_0,\label{eq:conv-Axk-b-im-x-ex-l}\\
		0\leqslant \mathcal L(x_{k},\lambda^*)-	\mathcal L(x^*,\lambda_{k})
		\leqslant \theta_k\mathcal E_0,
		\label{eq:conv-Lk-im-x-ex-l}
		\\	
		{}\snm{f(x_k)-f(x^*)}
		\leqslant  \theta_k\left(\mathcal E_0+\mathcal R_0\nm{\lambda^*}\right),
		\label{eq:conv-fk-im-x-ex-l}
	\end{subnumcases}
	where $\mathcal R_0$ has been defined by \cref{eq:R0} and 
\begin{equation}\label{eq:bk-est-im-x-ex-l}
	\theta_k
	\leqslant
	\min\left\{
	\frac{Q}{	\sqrt{\gamma_0}k+Q}
	,\,
	\frac{Q^2}{(\sqrt{\gamma_{\min}}k+Q)^2}
	\right\}\quad\text{with}\,\,Q = 3\nm{A}+\sqrt{\gamma_{\max}}.
\end{equation}
\end{thm}
\begin{proof}
	The fact $\{x_k\}\subset\mathcal X$ comes from \cref{eq:apd-im-x-ex-l-xk1-argmin}.	Following \cref{thm:conv-apd-im-x-im-l}, we start from the difference 
	$	\mathcal E_{k+1}-\mathcal E_{k} = I_1+I_2+I_3$, where $I_1,I_2$ and $I_3$ are defined in \cref{eq:I1-I2-I3}.
	
	For $I_2$, we continue with \cref{eq:I1-im-x-im-l} and insert $\whk$ into the last cross term to obtain
	\[
	\begin{aligned}
		I_2=&-\frac{\alpha_k\theta_{k+1}}{2}\nm{\lambda_{k+1}-\lambda^*}^2
		-\frac{\theta_{k}}{2}\nm{\lambda_{k+1}-\lambda_k}^2
		+\theta_{k}\big\langle \lambda_{k+1}-\lambda_k,\lambda_{k+1}-\whk+\whk-\lambda^*\big\rangle\\
		=&-\frac{\alpha_k\theta_{k+1}}{2}\nm{\lambda_{k+1}-\lambda^*}^2
		+\theta_{k}\big\langle \lambda_{k+1}-\lambda_k,\whk-\lambda^*\big\rangle
		+\frac{\theta_{k}}{2}\left(\big\|\lambda_{k+1}-\whk\big\|^2-\big\|\lambda_{k}-\whk\big\|^2\right)
		.
	\end{aligned}
	\]
	By \eqref{eq:apd-im-x-ex-l-l} we rewrite the cross term and 
	drop the negative term $-\big\|\lambda_{k}-\whk\big\|^2$ 
	to get
	\begin{equation}\label{eq:I2-im-x-ex-l}
		I_2\leqslant {}-\frac{\alpha_k\theta_{k+1}}{2}\nm{\lambda_{k+1}-\lambda^*}^2
		+\frac{\theta_{k}}{2}\big\|\lambda_{k+1}-\whk\big\|^2
		+\alpha_{k}\big\langle Av_{k+1}-b,\whk-\lambda^*\big\rangle.
	\end{equation}
	The estimation of $I_3$ is in line with that of  \cref{thm:conv-apd-im-x-im-l}, with $\lambda_{k+1}$ being $\whk$. For simplicity, we will not recast the redundant details here. Consequently, one finds that the estimate \cref{eq:I1+I2+I3-im-x-im-l} now becomes 
	\[
	\begin{aligned}
		\mathcal E_{k+1}-	\mathcal E_{k}\leqslant &-\alpha_{k}	\mathcal E_{k+1}
		+\frac{\theta_{k}}{2}\big\| \lambda_{k+1}-\whk\big\|^2
		-	\frac{\gamma_{k}}2\nm{v_{k+1}-v_k}^2.
	\end{aligned}
	\]
	Thanks to \eqref{eq:apd-im-x-ex-l-l} and \cref{eq:barl-k}, we have that $		\lambda_{k+1}-\whk = \alpha_k/\theta_{k}A(v_{k+1}-v_k)$, and by our choice $\gamma_{k}\theta_k=\nm{A}^2\alpha_k^2$, it is not hard to see 
	\[
	\begin{aligned}
		\frac{\theta_{k}}{2}\big\| \lambda_{k+1}-\whk\big\|^2
		\leqslant {}\frac{\nm{A}^2\alpha_k^2}{2\theta_{k}}\nm{v_{k+1}-v_k}^2
		=\frac{\gamma_{k}}2\nm{v_{k+1}-v_k}^2.
	\end{aligned}
	\]
	Putting this back to the previous estimate implies \cref{eq:diff-Ek-im-x-ex-l}.
	
	As the proof of \cref{eq:conv-Axk-b-im-x-ex-l} is similar with 
\cref{eq:conv-Lk-im-x-im-l}, it boils down to checking the 
decay estimate \cref{eq:bk-est-im-x-ex-l}. Let us start from the following estimate
\[
\frac{1}{\sqrt{\theta_{k+1}}}-\frac{1}{\sqrt{\theta_{k}}}
=\frac{\alpha_{k}/\sqrt{\theta_{k}} }{1+\sqrt{1+\alpha_{k}}}
=\frac{\sqrt{\gamma_k}}{\nm{A}+\sqrt{1+\alpha_k}\nm{A}},
\]
where we used the identity $\theta_{k}=\theta_{k+1}(1+\alpha_k)$ (cf. \cref{eq:apd-im-gamma-theta}) and the relation $\nm{A}^2\alpha_k^2 = \gamma_k\theta_{k}$. Since $\gamma_{\min}\leqslant \gamma_k\leqslant \gamma_{\max}$, we have
\[
\begin{aligned}
	\sqrt{1+\alpha_k}\nm{A}\leqslant {}&\nm{A}+\sqrt{\alpha_k}\nm{A}
	=\nm{A} + \sqrt{\nm{A}\sqrt{\theta_k\gamma_k}}\\
	\leqslant {}& \nm{A} + \nm{A}+\sqrt{\theta_{k}\gamma_k}
	\leqslant 2\nm{A}+\sqrt{\gamma_{\max}},
\end{aligned}
\]
and it follows that
\begin{equation}\label{eq:est-1}
	\frac{1}{\sqrt{\theta_{k+1}}}-\frac{1}{\sqrt{\theta_{k}}}\geqslant 
	\frac{\sqrt{\gamma_{\min}}}{Q}
	\quad\Longrightarrow\quad
	\theta_k\leqslant \frac{Q^2}{(\sqrt{\gamma_{\min}}k+Q)^2}.
\end{equation}
Here recall that $Q = 3\nm{A}+\sqrt{\gamma_{\max}}$.
In addition, by \eqref{eq:apd-im-gamma-theta}, we have
\[
\frac{\gamma_{k+1}}{\gamma_{k}} = \frac{1+\mu_\beta\alpha_k/\gamma_k}{1+\alpha_k}\geqslant \frac{1}{1+\alpha_k}=\frac{\theta_{k+1}}{\theta_{k}},
\]
which means $\gamma_k\geqslant \gamma_0\theta_k$ and also implies
\[
\frac{1}{\sqrt{\theta_{k+1}}}-\frac{1}{\sqrt{\theta_{k}}}
\geqslant 
\frac{\sqrt{\gamma_0}}{Q}\sqrt{\theta_{k}}.
\]
As $\theta_{k}\geqslant \theta_{k+1}>0$, we obtain 
\[
\frac{1}{\theta_{k+1}}-\frac{1}{\theta_{k}}\geqslant 
\frac{\sqrt{\gamma_0}}{Q}
\quad\Longrightarrow\quad
\theta_k\leqslant \frac{Q}{	\sqrt{\gamma_0}k+Q},
\]
which together with \cref{eq:est-1} gives \cref{eq:bk-est-im-x-ex-l}
and concludes the proof of this theorem.
\end{proof}

Analogously to \cref{eq:apd-im-x-im-l-lk1}, one has
\begin{subnumcases}{}
	\lambda_{k+1}={}\lambda_k-\frac{1}{\theta_{k}}(Ax_k-b)+\frac{1}{\theta_{k+1}}(Ax_{k+1}-b),
	\label{eq:apd-im-x-ex-l-lk1}\\
	x_{k+1}
	\in{}y_k
	-\eta_k\left(\partial f_\beta(x_{k+1})+N_{\mathcal X}(x_{k+1})+A^{\top}\whk\right),
	\label{eq:apd-im-x-ex-l-xk1}
\end{subnumcases}
where $\tau_k$ and $y_k$ are defined in \cref{eq:im-x-im-l-tauk-yk} and $\eta_k=\alpha_k^2/\tau_k$.
Then it is possible to eliminate $\lambda_{k+1}$ from \cref{eq:apd-im-x-ex-l-xk1} and get
\begin{equation}
	\label{eq:apd-im-x-ex-l-xk1-argmin}
	x_{k+1}
	=\mathop{ \argmin}\limits_{x\in\mathcal X}
	\left\{
	f_\beta(x)
	+\frac{1}{2\eta_k}\big\|x-y_k
	+\eta_kA^{\top}\whk\big\|^2
	\right\}.
\end{equation}
Comparing this with \cref{eq:apd-im-x-im-l-xk1-argmin}, we see explicit discretization of $\lambda$ in $\partial_x\mathcal L_\beta(x,\lambda)$ leads to linearization of the penalty term $\ell_{\theta_{k+1}}(x)$. As mentioned at the end of \cref{sec:apd-im}, the advantage of the augmented term in $f_\beta$ is to enlarge $\mu_\beta=\mu+\beta\sigma_{\min}^2(A)$ when $\sigma_{\min}(A)>0$. This promises $\gamma_{\min}>0$, and by \cref{eq:bk-est-im-x-ex-l} we have the faster rate $O(1/k^2)$ but the price is to compute $\prox_{f_\beta}^{\mathcal X}$. Otherwise, if $\sigma_{\min}(A)=0$, then that term is useless and we shall set $\beta=0$, which means \cref{eq:apd-im-x-ex-l-xk1-argmin} only involves the operation $\prox_{f}^{\mathcal X}$, i.e., the proximal computation of $f$ on $\mathcal X$.

To the end of this section, let us reformulate \cref{eq:apd-im-x-ex-l-l} with the step size $\gamma_{k}\theta_k=\nm{A}^2\alpha_k^2$ in \cref{algo:Semi-APD}, which is called the semi-implicit APD method.
\begin{algorithm}[H]
	\caption{Semi-implicit APD method for \cref{eq:min-f-X-Ax-b} with $f\in\mathcal S_{\mu}^{0}(\mathcal X), \mu\geqslant 0$.}
	\label{algo:Semi-APD}
	\begin{algorithmic}[1] 
		\REQUIRE  $\beta\geqslant 0,\,\theta_0=1,\,\gamma_0>0,\, (x_0,v_0)\in\mathcal X\times \R^n,\,\lambda_0\in\R^m$.
		\STATE Set $\beta=0$ if $\sigma_{\min}(A)=0$, and let $\mu_\beta=\mu+\beta\sigma_{\min}^2(A)$.
		\FOR{$k=0,1,\ldots$}
		\STATE Choose step size $\alpha_k=\sqrt{\theta_{k}\gamma_k}/\nm{A}$.
		\STATE Update $\displaystyle \gamma_{k+1} = (\gamma_k+\mu_\beta\alpha_k)/(1+\alpha_k)$ and $\displaystyle \theta_{k+1}= \theta_k/(1+\alpha_k)$.
		\STATE Set $\displaystyle  \tau_k = \gamma_k+\mu_\beta\alpha_k+\gamma_k\alpha_k$ and $\eta_k=\alpha_k^2/\tau_k$.
		\STATE Set $\displaystyle  y_k={}\left((\gamma_k
		+\mu_\beta\alpha_k)x_k+\gamma_k\alpha_kv_k\right)/\tau_k$.
		\STATE Compute $\whk
		=\lambda_k+\alpha_k/\theta_k\left(
		Av_k-b
		\right)$.
		\STATE Update $\displaystyle	
		x_{k+1}=\prox_{\eta_kf_\beta}^{\mathcal X}(y_k
		-\eta_kA^{\top}\whk)$.\label{step:Semi-APD-xk1}
		\STATE Update $\displaystyle v_{k+1} =x_{k+1}+(x_{k+1}-x_k)/\alpha_k$. 		
		\STATE Update $\lambda_{k+1}=
		\lambda_k+\alpha_k/\theta_k\left(
		Av_{k+1}-b
		\right)$.
		\ENDFOR
	\end{algorithmic}
\end{algorithm}
\begin{remark}
	Notice that for $\beta=0$, \cref{eq:apd-im-x-ex-l-xk1-argmin} 
	is close to the partially linearized proximal ALM. In addition, by using 
	the relation \eqref{eq:apd-im-x-ex-l-l}, we can drop the 
	sequence $\{\lambda_{k}\}$ and simplify \cref{algo:Semi-APD} 
	as a method involving only two-term sequence $\{(x_k,y_k)\}$. \hfill\ensuremath{\blacksquare}
\end{remark}
\begin{remark}
	From \cref{eq:conv-Lk-im-x-ex-l,eq:bk-est-im-x-ex-l}, we conclude the nonergodic convergence rate 
	\begin{equation}\label{eq:rate-im-x-ex-l}
		{}\snm{f(x_k)-f(x^*)}+\nm{Ax_k-b}	\leqslant 
		C
		\left\{
		\begin{aligned}
			&		\frac{\nm{A}}{k},&&\mu_\beta = 0,\\
			&	\frac{\nm{A}^2}{k^2},&&\mu _\beta> 0,
		\end{aligned}
		\right.
	\end{equation}
	where the implicit constant $C$ may depend on small $\gamma_0$.
	But for large $\gamma_0$ (compared with $\nm{A}$), $C$ can be uniformly bounded with respect to $\gamma_0$. This holds for all the rates in the sequel. 
	For a detailed verification of this claim, we refer to \cite{luo_differential_2021}.
	\hfill\ensuremath{\blacksquare}
\end{remark}
\begin{remark}\label{rem:jdmm}
	As mentioned at the end of \cref{sec:apd-im}, since the augmented term has been linearized, both the semi-implicit discretization \cref{eq:apd-im-x-ex-l-v} and the explicit one \cref{eq:apd-ex-x-ex-l-x-h-g} can be applied to the two block case \cref{eq:2b} directly.
	
	As a byproduct, the scheme \cref{eq:apd-im-x-ex-l-v} with $\beta=0$ leads to a linearized parallel proximal ADMM. Correspondingly, for updating $x_{k+1} = (x_{k+1}^1,x_{k+1}^2)$, step \ref{step:Semi-APD-xk1} of \cref{algo:Semi-APD} involves two parallel proximal calculations: $\prox_{\eta_kf_1}$ and $\prox_{\eta_kf_2}$. In fact, we claim that it can be extended to the multi-block case	
	\begin{equation}\label{eq:nb}
		f(x) =\sum_{i=1}^{n}f_i(x_i),\quad Ax = \sum_{i=1}^{n}A_ix_i,
	\end{equation}
	and the nonergodic rate \cref{eq:rate-im-x-ex-l} still holds true. This means for general convex $f_i$, we have the nonergodic rate $O(1/k)$ but to obtain the faster rate $O(1/k^2)$, all components $f_i$'s shall be strongly convex to ensure $\mu>0$. This is very close to the decomposition method in \cite{tran-dinh_primal-dual_2015} and the predictor corrector proximal multipliers \cite{chen_proximal-based_1994}. \qed
\end{remark}

\section{A Corrected Semi-implicit Operator Splitting Scheme}
\label{sec:apd-correc-ex-x-im-l}
The semi-implicit discretization proposed in \cref{sec:apd-ex-x-im-l} applies explicit discretization to $\lambda$ in \eqref{eq:apd-im-x-ex-l-v}. It is of course reasonable to use explicit discretization for $x$ in \eqref{eq:apd-im-x-ex-l-x}. To be more precise, consider the following  semi-implicit discretization for \cref{eq:apd-sys-lambda}:
\begin{subnumcases}{}
	\theta_k \frac{\lambda_{k+1}-\lambda_k}{\alpha_k} = {} \nabla_\lambda \mathcal L_\beta(v_{k+1},\lambda_{k+1}),\label{eq:apd-ex-x-im-l-l}\\
	\frac{x_{k+1}-x_{k}}{\alpha_k}={}v_{k}-x_{k+1},\label{eq:apd-ex-x-im-l-x}\\
	\gamma_k \frac{v_{k+1}-v_k}{\alpha_k} \in {}\mu_\beta(x_{k+1}-v_{k+1}) -\partial_x\mathcal L_\beta
	(x_{k+1},\lambda_{k+1}),
	\label{eq:apd-ex-x-im-l-v}
\end{subnumcases}
where the parameter system \cref{eq:theta-gamma} is 
still discretized by \cref{eq:apd-im-gamma-theta}.

As one may see, $x_{k+1}$ can be updated from \eqref{eq:apd-ex-x-im-l-x} easily but there comes a problem: can we compute the subgradient $\xi_{k+1}\in\partial f_{\beta}(x_{k+1})+N_{\mathcal X}(x_{k+1})$ ? Once such a $\xi_{k+1}$ is obtained, \eqref{eq:apd-ex-x-im-l-v} becomes
\begin{equation*}
	\gamma_k \frac{v_{k+1}-v_k}{\alpha_k} ={}\mu_\beta(x_{k+1}-v_{k+1}) -\big(\xi_{k+1}+A^{\top}\lambda_{k+1}\big).
\end{equation*}
Observing form this and \eqref{eq:apd-ex-x-im-l-l}, $\lambda_{k+1}$ is only linearly coupled with $v_{k+1}$.

However, to get $\xi_{k+1}$, we shall impose the condition: $x_{k+1}\in\mathcal X$, which is promised if both $x_k$ and $v_k$ belong to $\mathcal X$, as $x_{k+1}$ is a convex combination of them. Unfortunately, it is observed that the semi-implicit scheme \cref{eq:apd-ex-x-im-l-v} does not preserve the property: $v_{k+1}\in\mathcal X$. Therefore, the sequence $\{(x_k,v_k)\}$ may be outside $\mathcal X$.

Below, in \cref{sec:one-it-ex-x-im-l}, we shall give a one-iteration analysis to further illustrate the ``degeneracy'' of the scheme \cref{eq:apd-ex-x-im-l-l}, which loses the contraction property \cref{eq:diff-Ek-im-x-ex-l}, and then we propose a modified scheme as a remedy in \cref{sec:correc-ex-x-im-l}.
\subsection{A one-iteration analysis}
\label{sec:one-it-ex-x-im-l}
As before, we wish to establish the contraction property with respect to the discrete Lyapunov function \cref{eq:Ek} but there exists some cross term that makes us  in trouble.
\begin{lem}\label{thm:conv-apd-ex-x-im-l}
	Suppose $f\in\mathcal S_{\mu}^{0}(\mathcal X)$ with $\mu\geqslant 0$. Let $k\in\mathbb N$ be fixed and assume  $(x_k,v_k)\in \mathcal X\times \mathcal X$. Then for the semi-implicit scheme \cref{eq:apd-ex-x-im-l-l} with $\alpha_k>0$, we have $x_{k+1}\in\mathcal X$ and
	\begin{equation}\label{eq:diff-Ek-ex-x-im-l}
		\begin{aligned}
			\mathcal E_{k+1}\!-\!\mathcal E_{k}
			\leqslant&-\alpha_k\mathcal E_{k+1}	
			-\alpha_k
			\dual{ \xi_{k+1}+A^{\top}\lambda^*,v_{k+1}-v_k}
			-	\frac{\gamma_{k}}2\nm{v_{k+1}-v_k}^2
			-\frac{\theta_{k}}{2}\nm{\lambda_{k+1}-\lambda_{k}}^2,
		\end{aligned}
	\end{equation}
	where $\xi_{k+1}\in\partial f_\beta(x_{k+1})+N_{\mathcal X}(x_{k+1})$.
\end{lem}
\begin{proof}
	Again, let us follow the proof of \cref{thm:conv-apd-im-x-im-l} and begin with the difference 
	$	\mathcal E_{k+1}-\mathcal E_{k} = I_1+I_2+I_3$, where $I_1,I_2$ and $I_3$ are defined in \cref{eq:I1-I2-I3}. 
	
	We just copy the identity \cref{eq:I1-im-x-im-l-simple} for $I_2$ here:
	\[
	I_2=-\frac{\alpha_k\theta_{k+1}}{2}\nm{\lambda_{k+1}-\lambda^*}^2
	-\frac{\theta_{k}}{2}\nm{\lambda_{k+1}-\lambda_{k}}^2
	+\alpha_{k}\dual{Av_{k+1}-b,\lambda_{k+1}-\lambda^*}.
	\]
	For $I_3$, let us start from \cref{eq:I2-2-im-x-ex-l}, i.e.,
	\[
	\begin{aligned}
		I_3
		= {}&\mu_\beta\alpha_k\dual{x_{k+1} - v_{k+1}, v_{k+1} - x^{*}}
		-
		\alpha_k \dual{ \xi_{k+1}+A^{\top}\lambda^*, v_{k+1} - x^{*}}\\
		{}&\quad+	\frac{\alpha_k(\mu_\beta-\gamma_{k+1})}{2}
		\nm{v_{k+1}-x^*}^2		-	\frac{\gamma_{k}}2\nm{v_{k+1}-v_k}^2\\
		{}&\qquad	-\alpha_k\dual{Av_{k+1}-b,\lambda_{k+1}-\lambda^*} .
	\end{aligned}
	\]
	The first cross term is expanded as \cref{eq:cross} but the second cross term contains more:
	\begin{align*}
		{}&
		-\alpha_k \dual{ \xi_{k+1}+A^{\top}\lambda^*, v_{k+1} - x^{*}} \\
		=&-\alpha_k
		\dual{ \xi_{k+1}+A^{\top}\lambda^*,v_{k+1}-v_k}
		-	\dual{ \xi_{k+1}+A^{\top}\lambda^*,x_{k+1}-x_k} \\
		{}&\quad	-\alpha_k\dual{\xi_{k+1}+A^{\top}\lambda^*, x_{k+1}-x^*},
	\end{align*}
	where we have used \eqref{eq:apd-ex-x-im-l-x}. Similar with \cref{eq:est-mid-I3}, we have 
	\[
	\begin{aligned}
		-\alpha_k \dual{ \xi_{k+1}+A^{\top}\lambda^*, v_{k+1} - x^{*}} 
		\leqslant {}&\mathcal L_\beta(x_k,\lambda^*)-\mathcal L_\beta(x_{k+1},\lambda^*)
		-\alpha_k
		\dual{ \xi_{k+1}+A^{\top}\lambda^*,v_{k+1}-v_k}\\
		{}&\quad+\alpha_k\left(\mathcal L_\beta(x^*,\lambda^*)-\mathcal L_\beta(x_{k+1},\lambda^*)\right)
		-\frac{\mu_\beta\alpha_k}{2}\nm{x_{k+1}-x^*}^2.
	\end{aligned}
	\]
	Note that $I_1$ and the first term in the above estimate cancel out each other.
	Summarizing those results, we find that \cref{eq:I1+I2+I3-im-x-im-l} now reads as 
	\begin{equation}
		\label{eq:diff-ex-x-im-l-mid}
		\begin{aligned}
			\mathcal E_{k+1}\!-\!	\mathcal E_{k}\leqslant &-\alpha_{k}	\mathcal E_{k+1}
			-\alpha_k
			\dual{ \xi_{k+1}+A^{\top}\lambda^*,v_{k+1}-v_k}
			-	\frac{\gamma_{k}}2\nm{v_{k+1}-v_k}^2
			-\frac{\theta_{k}}{2}\nm{\lambda_{k+1}-\lambda_{k}}^2,
		\end{aligned}
	\end{equation}
	which gives 
	\cref{eq:diff-Ek-ex-x-im-l} and completes the proof of this lemma.
\end{proof}
\subsection{Correction via extrapolation}
\label{sec:correc-ex-x-im-l}
We now have two main difficulties: one is to cancel the cross terms in \cref{eq:diff-Ek-ex-x-im-l}, and the other is to maintain the sequence $\{(x_k,v_k)\}$ in $\mathcal X$. For the first, following the main idea from \cite{luo_differential_2021}, we replace $x_{k+1}$ in \cref{eq:apd-ex-x-im-l-x} by $y_k$ and add an extra extrapolation step to update $x_{k+1}$. For the second, a minor modification is to substitute $\partial_x\mathcal L_\beta
(x_{k+1},\lambda_{k+1}) $ in \eqref{eq:apd-ex-x-im-l-v} with 
$\partial_x\mathcal L_\beta
(v_{k+1},\lambda_{k+1}) $ and this leads to
\begin{subnumcases}{}
	\theta_k \frac{\lambda_{k+1}-\lambda_k}{\alpha_k} = {} \nabla_\lambda \mathcal L_\beta(v_{k+1},\lambda_{k+1}),\label{eq:apd-ex-x-im-l-l-f}\\
	\frac{y_{k}-x_{k}}{\alpha_k}={}v_{k}-y_{k},\label{eq:apd-ex-x-im-l-y-f}\\
	\gamma_k \frac{v_{k+1}-v_k}{\alpha_k} \in {}\mu_\beta(y_{k}-v_{k+1}) - 
	\partial_x\mathcal L_\beta
	(v_{k+1},\lambda_{k+1}) ,
	\label{eq:apd-ex-x-im-l-v-f}\\
	\frac{x_{k+1}-x_{k}}{\alpha_k}={}v_{k+1}-x_{k+1}.
	\label{eq:apd-ex-x-im-l-x-f}
\end{subnumcases}
Here the step \eqref{eq:apd-ex-x-im-l-v-f} becomes implicit, i.e., $f_\beta$ is discretized implicitly in terms of $v_{k+1}$. Although \cref{eq:apd-ex-x-im-l-x-f} is totally different from the fully implicit method \cref{eq:apd-im-x-im-l-v} and the previous semi-implicit method \cref{eq:apd-im-x-ex-l-v}, both of which applied implicit discretization to $f_\beta$ (with respect to $x_{k+1}$), we shall leave it alone and adopt possible explicit discretization for $f_\beta$. This is somewhat equivalent to linearizing $f_\beta$ and thus requires smoothness of $f$. 

Therefore, in general, we consider the composite 
case $f = h+g$ where $h\in\mathcal S_{\mu,L}^{1,1}(\mathcal X)$ with $0\leqslant \mu\leqslant L<\infty$ 
and $g\in\mathcal S_0^0(\mathcal X)$. Then linearization can be applied to the smooth part $h$ while implicit scheme is maintained for the nonsmooth part $g$. This utilizes the separable structure of $f$ and is called {\it operator splitting}, which is also known as forward-backward technique. Needless to say, the case $g=0$ is allowed, and for $h\in\mathcal S_{0,L}^{1,1}(\mathcal X),\,g\in\mathcal S_\mu^0(\mathcal X)$, we can split $h+g$ as $(h(x) + \mu/2\|x\|^2) + (g(x) - \mu/2\|x\|^2)$, which reduces to our current setting.

Keeping this in mind, we consider the following corrected semi-implicit scheme: given $(\lambda_k,x_k,v_k)\in\R^m\times \mathcal X\times \mathcal X$ and $\alpha_k>0$, compute $(\lambda_{k+1},x_{k+1},v_{k+1})\in\R^m\times \mathcal X\times \mathcal X$ from 
\begin{subnumcases}{}
	\theta_k \frac{\lambda_{k+1}-\lambda_k}{\alpha_k} = {} \nabla_\lambda \mathcal L_\beta(v_{k+1},\lambda_{k+1}),\label{eq:apd-ex-x-im-l-l-h-g}\\
	\frac{y_{k}-x_{k}}{\alpha_k}={}v_{k}-y_{k},\label{eq:apd-ex-x-im-l-y-h-g}\\
	\gamma_k \frac{v_{k+1}-v_k}{\alpha_k} \in \mu_\beta(y_{k}- v_{k+1}) - 
	\left(\nabla h_\beta(y_k)+ \partial g_{\mathcal X}(v_{k+1})+ A^{\top}\lambda_{k+1}\right),
	\label{eq:apd-ex-x-im-l-v-h-g}\\
	\frac{x_{k+1}-x_{k}}{\alpha_k}={}v_{k+1}-x_{k+1},
	\label{eq:apd-ex-x-im-l-x-h-g}
\end{subnumcases}
where $g_{\mathcal X} = g+\delta_{\mathcal X}$ and $\partial g_{\mathcal X}(v_{k+1})=\partial g(v_{k+1})+N_{\mathcal X}(v_{k+1})$.

Evidently, the step \eqref{eq:apd-ex-x-im-l-v-h-g} can be rewritten as
\begin{equation}\label{eq:apd-ex-x-im-l-lk1-h-g-argmin}
	v_{k+1} = \mathop{\argmin}_{\mathcal X}\left\{
	g(v)+\dual{A^{\top}\lambda_{k+1}+	\nabla h_\beta(y_k),v}+\frac{\tau_k}{2\alpha_k}\nm{v-w_k}^2
	\right\},
\end{equation}
where $\tau_k=\gamma_{k}+\mu_\beta\alpha_k$ and $w_{k}={}(\gamma_kv_{k}+\mu_\beta\alpha_ky_k)/\tau_k$.
Also, after eliminating $\lambda_{k+1}$, \cref{eq:apd-ex-x-im-l-lk1-h-g-argmin} 
can be further rearranged as follows
\begin{equation}\label{eq:apd-ex-x-im-l-lk1-h-g-argmin-new}
	v_{k+1} = \mathop{\argmin}_{\mathcal X}\left\{
	g(v)+\dual{z_k,v}+\frac{\alpha_k}{2\theta_{k}}\nm{Av-b}^2
	+\frac{\tau_k}{2\alpha_k}\nm{v-w_k}^2
	\right\},
\end{equation}
where $z_k=\nabla h_\beta(y_k)+A^{\top}\lambda_{k}$. Since $x_k,\,v_k\in\mathcal X$, by \eqref{eq:apd-ex-x-im-l-y-h-g} it clear that $y_k\in\mathcal X$, and once $v_{k+1}\in\mathcal X$ is obtained, we can update $x_{k+1}\in\mathcal X$ and $\lambda_{k+1}$ sequentially. Whence, if $x_0,v_0\in \mathcal X$, then the modified scheme \cref{eq:apd-ex-x-im-l-l-h-g} maintains $\{(x_k,y_k,v_k)\}\subset\mathcal X$.

Particularly, if $\mathcal X=\R^n$, then the step \cref{eq:apd-ex-x-im-l-lk1-h-g-argmin-new} is very close to \cite[Algorithm 3]{he_convergence_2021} and the accelerated linearized proximal ALM \cite{Xu2017}, both of which are proved to 
possess the nonergodic rate $O(L/k^2)$ under the assumption that $f=h+g$ is convex and $h$ has $L$-Lipschitz continuous gradient. As proved below in \cref{thm:apd-ex-x-im-l-h-g}, our method \cref{eq:apd-ex-x-im-l-x-h-g} also enjoys this rate for $\mu_\beta=0$. But for $\mu_\beta>0$, we have faster linear rate, and in \cref{sec:algo-sub-ex-x-im-l}, following the spirit from \cite{li_asymptotically_2020,luo_primal-dual_2021,niu_sparse_2021}, we will discuss how to design proper inner solver by utilizing the structure of the subproblem with respect to $\lambda_{k+1}$, instead of computing $v_{k+1}$ directly from \cref{eq:apd-ex-x-im-l-lk1-h-g-argmin-new}.
\subsection{Nonergodic convergence rate}
In this part, let us establish the contraction property of the corrected semi-implicit 
scheme \cref{eq:apd-ex-x-im-l-l-h-g} and prove its convergence rate.
\begin{thm}\label{thm:apd-ex-x-im-l-h-g}
	Assume $f=h+g$ where $h\in\mathcal S_{\mu,L}^{1,1}(\mathcal X)$ with $0\leqslant \mu\leqslant L<\infty$ and $g\in\mathcal S_0^{0}(\mathcal X)$. Given initial value $x_0,v_0\in \mathcal X$, the corrected semi-implicit scheme \cref{eq:apd-ex-x-im-l-x-h-g} generates $\{(x_k,y_k,v_k)\}\subset \mathcal X$, and if  $L_\beta\alpha_k^2=\gamma_{k}$, then there holds 
	\begin{equation}\label{eq:diff-Ek-ex-x-im-l-correc}
		\mathcal E_{k+1}-	\mathcal E_k
		\leqslant -\alpha_k		\mathcal E_{k+1}, \quad\text{for all~}k\in\mathbb N,
	\end{equation}
	which implies that
	\begin{subnumcases}
		{}\nm{Ax_{k}-b}\leqslant \theta_k\mathcal R_0,\label{eq:conv-Axk-b-ex-x-im-l}\\
		0\leqslant \mathcal L(x_{k},\lambda^*)-	\mathcal L(x^*,\lambda_{k})
		\leqslant \theta_k\mathcal E_0,
		\label{eq:conv-Lk-ex-x-im-l}
		\\	
		{}\snm{f(x_k)-f(x^*)}
		\leqslant  \theta_k\left(\mathcal E_0+\mathcal R_0\nm{\lambda^*}\right),
		\label{eq:conv-fk-ex-x-im-l}
	\end{subnumcases}
	where $\mathcal R_0$ has been defined by \cref{eq:R0} and 
	\begin{equation}\label{eq:bk-est-ex-x-im-l}
		\theta_k\leqslant 
		\min\left\{
		\frac{4L_\beta}{(\sqrt{\gamma_0}\, k+2\sqrt{L_\beta})^2},\,
		\left(1+\sqrt{\frac{\gamma_{\min}}{L_\beta}}\right)^{-k}
		\right\}.
	\end{equation}
	Here, recall that $\gamma_{\min} = \min\{\mu_\beta,\gamma_{0}\}$.
\end{thm}
\begin{proof}
	The fact $\{(x_k,y_k,v_k)\}\subset \mathcal X$ has been showed above. As before, we focus on the difference 
	$	\mathcal E_{k+1}-\mathcal E_{k} = I_1+I_2+I_3$, where $I_1,I_2$ and $I_3$ are defined in \cref{eq:I1-I2-I3}.
	
	For the first term $I_1$, we have
	\begin{equation}
		\label{eq:I1-ex-x-im-l-h-g}
		I_1 = h_\beta(x_{k+1})-h_\beta(x_k)+ g(x_{k+1})-g(x_k)+\dual{\lambda^*,A(x_{k+1}-x_k)},
	\end{equation}
	and the identity \cref{eq:I1-im-x-im-l-simple} for $I_2$ holds true here.
	
	For $I_3$, we shall begin with \cref{eq:I2-1-im-x-ex-l}, i.e.,
	\[
	\begin{aligned}
		I_3
		= {}&\frac{\alpha_k(\mu_\beta-\gamma_{k+1})}{2}
		\nm{v_{k+1}-x^*}^2
		-	\frac{\gamma_{k}}2\nm{v_{k+1}-v_k}^2
		+\gamma_k\dual{v_{k+1} - v_{k}, v_{k+1} - x^{*}}.
	\end{aligned}
	\]
	From \cref{eq:apd-ex-x-im-l-lk1-h-g-argmin}, it is not hard to obtain the necessary optimal condition of $v_{k+1}$ (see \cite[Eq (2.9) for instance]{nesterov_gradient_2013}):
	\[
	\dual{\frac{\tau_k}{\alpha_k}(v_{k+1}-w_k)+\zeta_{k+1}+A^{\top}\lambda_{k+1}+	\nabla f_\beta(y_k),\,v_{k+1}-x}\leqslant 0,
	\]
	for all $x\in\mathcal X$, where $\tau_k=\gamma_k+\mu_\beta\alpha_k,\,w_{k}={}(\gamma_kv_{k}+\mu_\beta\alpha_ky_k)/\tau_k$ and $\zeta_{k+1}\in\partial g(v_{k+1})$. Particularly, we have
	\[
	\begin{aligned}
		\gamma_{k}\dual{v_{k+1}-v_k,v_{k+1}-x^*}
		\leqslant
		{}&\mu_\beta\alpha_k\dual{y_k-v_{k+1},v_{k+1}-x^*}\\
		{}&\,	-\alpha_k\dual{\zeta_{k+1}+A^{\top}\lambda_{k+1}+\nabla h_\beta(y_k),v_{k+1}-x^*}.
	\end{aligned}
	\]
	By \cref{eq:3id}, the first cross term in the above inequality is rewritten as follows
	\[
	\begin{aligned}
		\mu_\beta\alpha_k	\dual{y_{k} - v_{k+1}, v_{k+1} - x^{*}}
		={}&\frac{\mu_\beta\alpha_k}{2}
		\left(\left\|y_{k}-x^{*}\right\|^{2}
		-\|y_{k}-v_{k+1}\|^{2}
		-\left\|v_{k+1}-x^{*}\right\|^{2}\right)\\
		\leqslant{}&\frac{\mu_\beta\alpha_k}{2}
		\left(\left\|y_{k}-x^{*}\right\|^{2}
		-\left\|v_{k+1}-x^{*}\right\|^{2}\right).
	\end{aligned}
	\]
	Thanks to the extrapolation step \eqref{eq:apd-ex-x-im-l-x-h-g}, it holds that
	\begin{equation*}
		\begin{aligned}
			&	-\alpha_k\dual{ A^{\top}\lambda_{k+1},v_{k+1}-x^*}
			=-\alpha_k\dual{\lambda_{k+1},Av_{k+1}-b}\\
			=&-\alpha_k\dual{\lambda_{k+1}-\lambda^*,Av_{k+1}-b}
			-\alpha_k\dual{\lambda^*,Av_{k+1}-b}\\
			=&-\alpha_k\dual{\lambda_{k+1}-\lambda^*,Av_{k+1}-b}
			-\alpha_k\dual{\lambda^*,Ax_{k+1}-b}-\dual{\lambda^*,A(x_{k+1}-x_k)}.
		\end{aligned}
	\end{equation*}
	This together with the convexity of $g$ and the fact $\{v_k\}\subset \mathcal X$ gives 
	\[
	\begin{aligned}
		{}&
		-\alpha_k \dual{\zeta_{k+1}+A^{\top}\lambda_{k+1}, v_{k+1} - x^{*}}\\
		\leqslant &-\alpha_k\dual{\lambda_{k+1}-\lambda^*,Av_{k+1}-b}
		-\alpha_k\dual{\lambda^*,Ax_{k+1}-b}-\dual{\lambda^*,A(x_{k+1}-x_k)}\\
		{}&\quad-\alpha_k(g(x_{k+1})-g(x^*))-\alpha_k(g(v_{k+1})-g(x_{k+1})).
	\end{aligned}
	\]
	According to the update for $y_{k}$ (cf. \eqref{eq:apd-ex-x-im-l-y-h-g}), we find
	\begin{align*}
		&
		-\alpha_k \dual{ \nabla h_\beta(y_k), v_{k+1} - x^{*}} \\
		=&-\alpha_k\dual{\nabla h_\beta(y_k),v_{k+1} - v_k} 
		-\alpha_k\dual{  \nabla h_\beta(y_k), v_{k} - x^{*}}\\
		=&-\alpha_k\dual{\nabla h_\beta(y_k),v_{k+1} - v_k}
		- \dual{ \nabla h_\beta(y_k), y_{k} - x_{k}} 
		- \alpha_k\dual{  \nabla h_\beta(y_k), y_{k} - x^{*}}.
	\end{align*}
	As $h_\beta\in\mathcal S_{\mu_\beta}^1(\mathcal X)$, by the fact $\{(x_k,y_k)\}\subset \mathcal X$, it follows that
	\begin{align*}
		{}&
		- \dual{ \nabla h_\beta(y_k), y_{k} - x_{k}} 
		- \alpha_k\dual{  \nabla h_\beta(y_k), y_{k} - x^{*}}\\
		\leqslant {}&h_\beta(x_k)-h_\beta(y_k)
		- \alpha_k\left(h_\beta(y_k)-h_\beta(x^*)\right) 
		-\frac{\mu_\beta\alpha_k}{2}\nm{x^*-y_k}^2\\
		={}&h_\beta(x_{k})-h_\beta(x_{k+1}) +
		(1+\alpha_k)\left(h_\beta(x_{k+1})-h_\beta(y_{k})\right)
		\\		
		{}&\quad	- \alpha_k\left(h_\beta(x_{k+1})-h_\beta(x^*)\right) 
		-\frac{\mu_\beta\alpha_k}{2}\nm{x^*-y_k}^2.
	\end{align*}
	
	Hence, summarizing the above detailed expansions yields the estimate of $I_3$ and by a careful but not hard rearrangement of all the bounds from $I_1$ to $I_3$, we arrive at
	\begin{equation}\label{eq:diff-ex-x-im-l-h-g}
		\begin{aligned}
			\mathcal E_{k+1}-\mathcal E_{k}\leqslant &
			-\alpha_k\mathcal E_{k+1}
			-	\frac{\gamma_{k}}2\nm{v_{k+1}-v_k}^2
			-\frac{\theta_{k}}{2}\nm{\lambda_{k+1}-\lambda_{k}}^2\\
			{}&		+(1+\alpha_k)\left(h_\beta(x_{k+1})- h_\beta(y_{k})\right)	-\alpha_k\dual{\nabla h_\beta(y_k),v_{k+1} - v_k}\\
			{}&	\quad			+(1+\alpha_k)g(x_{k+1})-g(x_{k})	-\alpha_kg(v_{k+1}).		
		\end{aligned}
	\end{equation}
	Recalling \eqref{eq:apd-ex-x-im-l-x-h-g}, $x_{k+1}$ is a convex combination of $x_k$ and $v_{k+1}$ and the last line of \cref{eq:diff-ex-x-im-l-h-g} is nonpositive.
	Let us consider the second line. It is clear that \cite[Chapter 2]{nesterov_lectures_2018}
	\[
	h_\beta(x_{k+1})-h_\beta(y_k)\leqslant \dual{\nabla h_\beta(y_k),x_{k+1}-y_k} + \frac{L_\beta}{2}\nm{x_{k+1}-y_k}^2.
	\]
	From \eqref{eq:apd-ex-x-im-l-y-h-g} and \eqref{eq:apd-ex-x-im-l-x-h-g}, we obtain the relation $(1+\alpha_k)	(x_{k+1}-y_k) = \alpha_k(v_{k+1}-v_k)$,
	which together with the previous estimate gives
	\begin{equation}
		\label{eq:fk1-fk-ex-x-im-l}
		\begin{aligned}
			(1+\alpha_k)\left(h_\beta(x_{k+1})-h_\beta(y_{k})\right)
			-\alpha_k\dual{\nabla h_\beta(y_k),v_{k+1} - v_k}
			\leqslant 
			{}&
			\frac{L_\beta\alpha_k^2}{2(1+\alpha_k)}\nm{v_{k+1}-v_k}^2.
		\end{aligned}
	\end{equation}
	Plugging this into \cref{eq:diff-ex-x-im-l-h-g} gives
	\[
	\mathcal E_{k+1}-\mathcal E_{k}\leqslant 
	-\alpha_k\mathcal E_{k+1}
	+\left(\frac{L_\beta\alpha_k^2}{2(1+\alpha_k)}
	-\frac{\gamma_k}{2}\right)
	\nm{v_{k+1}-v_k}^2
	\leqslant 
	-\alpha_k\mathcal E_{k+1},
	\]
	where we have used the relation $L_\beta\alpha_k^2=\gamma_{k}$.
	The above estimate  implies \cref{eq:diff-Ek-ex-x-im-l-correc}.
	
	The proof of \cref{eq:conv-fk-ex-x-im-l} is analogous to that of \cref{eq:conv-Axk-b-im-x-im-l}. Clearly, we have $\alpha_{k}=\sqrt{\gamma_{k}/L_\beta}\geqslant \sqrt{\gamma_{\min}/L_\beta}$. If $\mu_\beta=0$, then $\{\gamma_k\}$ and $\{\theta_k\}$ are equivalent in the sense that $\gamma_k = \gamma_0\theta_k$. Therefore, a similar discussion as that of \cref{eq:bk-est-im-x-ex-l} establishes the decay estimate \cref{eq:bk-est-ex-x-im-l}.
	This completes the proof of this theorem.
\end{proof}
\subsection{Main algorithm and its subproblem}
\label{sec:algo-sub-ex-x-im-l}
Let us reformulate \eqref{eq:apd-ex-x-im-l-l-h-g} and \cref{eq:apd-ex-x-im-l-lk1-h-g-argmin} as follows
\begin{equation}\label{eq:sub-ex-x-im-l-h-g}
	\left\{
	\begin{aligned}
		\lambda_{k+1}={}&\lambda_k+\alpha_k/\theta_k(Av_{k+1}-b),\\
		v_{k+1}	 ={} &\prox_{t_k g}^{\mathcal X}(z_k-t_kA^{\top}\lambda_{k+1}),
	\end{aligned}
	\right.
\end{equation}
where $w_{k}$ and $\tau_k$ are the same in \cref{eq:apd-ex-x-im-l-lk1-h-g-argmin}, $t_k = \alpha_k/\tau_k$ and $z_k = w_k-t_k\nabla h_\beta(y_k)$. 
In the sequel, we shall discuss how to solve the subproblem \cref{eq:sub-ex-x-im-l-h-g} by 
utilizing its special structure. In summary, there are two cases. The first one $g = 0$ and $\mathcal X=\R^n$ leads to a linear saddle point system \cref{eq:lk1-vk1-v-X=R} and further gives two SPD systems \cref{eq:lk1-SPD,eq:vk1-SPD}, both of which can be solved via PCG (\cref{algo:PCG}). For the rest general case, \cref{eq:sub-ex-x-im-l-h-g} is transformed into a nonlinear equation (cf. \cref{eq:lk1-nonlinear}) in terms of  $\lambda_{k+1}$ and it is possible to use the SsN method (\cref{algo:SsN}), which would be quite efficient provided that the problem itself is semismooth and has sparsity structure.

We put detailed discussions of the  subproblem \cref{eq:sub-ex-x-im-l-h-g} in the following part and summarize the corrected scheme \cref{eq:apd-ex-x-im-l-l-h-g} with the step size $L_\beta\alpha_k^2=\gamma_k$ in \cref{algo:Semi-APDFB-h-g}, which is called the semi-implicit accelerated primal-dual forward-backward (Semi-APDFB for short) method.
\begin{algorithm}[h]
	\caption{Semi-APDFB method for $f = h+g$ with $h\in\mathcal S_{\mu,L}^{1,1}(\mathcal X)$ and $g\in\mathcal S_0^0(\mathcal X)$}
	\label{algo:Semi-APDFB-h-g}
	\begin{algorithmic}[1] 
		\REQUIRE  $\beta\geqslant 0,\,\theta_0=1,\,\gamma_0>0,\, (x_0,v_0)\in\mathcal X\times \mathcal X,\,\lambda_0\in\R^m$.
		\STATE Set $\beta=0$ if $\sigma_{\min}(A)=0$.
		\STATE Let $\mu_\beta=\mu+\beta\sigma_{\min}^2(A)$ 
		and $L_\beta=L+\beta\nm{A}^2$.
		\FOR{$k=0,1,\ldots$}
		\STATE Choose step size $\alpha_k=\sqrt{\gamma_k/L_\beta}$.
		\STATE Update $\displaystyle \gamma_{k+1} = (\gamma_k+\mu_\beta\alpha_k)/(1+\alpha_k)$ and $\displaystyle \theta_{k+1}= \theta_k/(1+\alpha_k)$.
		\STATE Set $\displaystyle  \tau_k = \gamma_k+\mu_\beta\alpha_k$ and $\displaystyle  y_{k}={}(x_k+\alpha_kv_{k})/(1+\alpha_k)$.
		\STATE Set  $		w_{k}={}\tau_k^{-1}(\gamma_kv_{k}+\mu_\beta\alpha_ky_k)$ and $z_k = w_k-\alpha_k/\tau_k\nabla h_\beta(y_k)$.
		\IF{$g=0$ and $\mathcal X=\R^n$}
		\STATE Solve $(\lambda_{k+1},\,v_{k+1})$ from the linear saddle-point system \cref{eq:lk1-vk1-l-X=R}, which can be done by applying  \cref{algo:PCG} to either \cref{eq:lk1-SPD} or \cref{eq:vk1-SPD} with  suitable preconditioner $M$ and the tolerance $\varepsilon$.
		\ELSE
		\STATE Solve $\lambda_{k+1}$ from the nonlinear equation \cref{eq:lk1-nonlinear} via \cref{algo:SsN}.
		\STATE Update $\displaystyle  v_{k+1}={}\prox^{\mathcal X}_{t_k g}(z_k-t_kA^{\top}\lambda_{k+1})$.		
		\ENDIF
		\STATE Update $\displaystyle  x_{k+1}={}(x_k+\alpha_kv_{k+1})/(1+\alpha_k)$.
		\ENDFOR
	\end{algorithmic}
\end{algorithm}
\subsubsection{The case $g = 0$ and $\mathcal X=\R^n$}
Let us first treat this special case and take the opportunity here to present a practical PCG method. In this situation, \cref{eq:sub-ex-x-im-l-h-g} reads simply as follows
\begin{subnumcases}{}
	v_{k+1}	 ={}  z_k-t_kA^{\top}\lambda_{k+1},
	\label{eq:lk1-vk1-v-X=R}\\
	\lambda_{k+1}=\lambda_k+\alpha_k/\theta_k(Av_{k+1}-b).
	\label{eq:lk1-vk1-l-X=R}
\end{subnumcases}
Eliminating $v_{k+1}$ gives
\begin{equation}\label{eq:lk1-SPD}
	\left(\theta_kI+\alpha_kt_kAA^{\top}\right)\lambda_{k+1} 
	= \theta_{k}\lambda_k+\alpha_k(Az_k-b).
\end{equation}
On the other hand, we have
\begin{equation}\label{eq:vk1-SPD}
	\left(\theta_kI+\alpha_kt_kA^{\top}A\right)v_{k+1} 
	=\theta_kz_k-t_k
	A^{\top}(\theta_k\lambda_k-\alpha_kb).
\end{equation}

Practically, we can choose the one with smaller size and consider suitable efficient linear SPD solvers. In \cref{algo:PCG}, 
we present a practical PCG iteration (cf. \cite[Appendix B3]{shewchuk_introduction_1994}) for solving a given SPD system $Hd = e$ with the tolerance $\varepsilon$ and the preconditioner $M$ that is an SPD approximation of $H$ and easy to invert.
\begin{algorithm}[h]
	\caption{A Practical PCG for the SPD system $Hd = e$}
	\label{algo:PCG}
	\begin{algorithmic}[1] 
		\REQUIRE  ~~\\
		$H$: an SPD matrix, $M$: the preconditioner;\\
		$e$: the right hand side vector, $\varepsilon\in(0,1)$: the error tolerance.
		\ENSURE An approximation $d$ to $H^{-1} e$.
		\STATE Choose an initial guess $d_0$.
		\STATE Set maximum number of iterations 
		$i_{\max}$.
		\STATE $i = 0$,\,$r= e - Hd_0,\,p= M^{-1}r,\,\delta= \dual{r,p},\,\delta_0=\delta$.
		\WHILE{$i<i_{\max}$ and $\delta>\varepsilon^2\delta_0$}
		\STATE $\delta_{\rm old}=\delta,\,q = Hp,\,\alpha = \delta_{\rm old}/\dual{q,p},\,d = d + \alpha p$.
		\IF{$i$ is divisible by 50} 
		\STATE $r = e-H d$.
		\ELSE
		\STATE $r = r-\alpha q$.
		\ENDIF		
		\STATE $w = M^{-1}r,\,\delta = \dual{r,w},\,
		\beta= \delta/\delta_{\rm old},\,
		p = w + \beta p$.		
		\STATE $i = i+1$.
		\ENDWHILE
	\end{algorithmic}
\end{algorithm}
\subsubsection{The general case}
\label{sec:SsN}
Now, introduce a mapping $F_k:\R^m\to\R^m$ by that
\begin{equation*}
	F_k(\lambda) := 	\theta_{k}\lambda-\alpha_k A\prox_{t_k g}^{\mathcal X}(z_k-t_kA^{\top}\lambda) - r_k,
\end{equation*}
where $r_k = \theta_{k}\lambda_k-\alpha_kb$.
Then eliminating $v_{k+1}$ from \cref{eq:sub-ex-x-im-l-h-g}
gives a nonlinear equation
\begin{equation}
	\label{eq:lk1-nonlinear}
	F_k(\lambda_{k+1}) =0.
\end{equation}

Note that $\prox_{t_kg}^{\mathcal X}$ is nothing but the proximal operator of $g_{\mathcal X}=g+\delta_{\mathcal X}$. Hence, it is monotone and $1$-Lipschitz continuous. In fact, we have (cf. \cite[Proposition 12.27]{Bauschke2011})
\begin{equation*}
	\dual{\prox_{t_kg}^{\mathcal X}(x)-\prox_{t_kg}^{\mathcal X}(y),x-y}
	\geqslant \nm{\prox_{t_kg}^{\mathcal X}(x)-\prox_{t_kg}^{\mathcal X}(y)}^2,
\end{equation*}
for all $(x,y)\in\R^n\times \R^n$, which implies
\begin{equation}\label{eq:Lip-Fk-ex-x-im-l}
	\theta_{k}\nm{\lambda-\xi}^2\leqslant 
	\dual{F_k(\lambda)-F_k(\xi),\lambda-\xi}
	\leqslant
	\rho_{k}\nm{\lambda-\xi}^2
	\quad\forall\,(\lambda,\xi)\in\R^m\times \R^m,
\end{equation}
where $\rho_{k}=\theta_{k}+\alpha_kt_k\nm{A}^2$. 
Therefore $F_k$ is monotone and $\rho_{k}$-Lipschitz continuous. 

As conventional, denote by $g_{\mathcal X}^*$ the conjugate function of $g_{\mathcal X}$ and introduce
\begin{equation}\label{eq:cal-Fk}
	\mathcal F_k(\lambda): = 		\frac{\theta_{k}}{2}\nm{\lambda}^2-\dual{r_k,\lambda}
	+\alpha_k[g_{\mathcal X}^*]_{t_k} (z_k/t_k-A^{\top}\lambda),\quad 
	\text{ for all }\lambda\in\R^m,
\end{equation}
where for any $t>0,\,[g_{\mathcal X}^*]_{t}:\R^n\to\R^n$ stands for the Moreau--Yosida approximation of $g_{\mathcal X}^*$ with parameter $t>0$, i.e.,
\[
[g_{\mathcal X}^*]_{t}(x): = \min_{y\in\R^n}
\left\{
g_{\mathcal X}^*(y) + \frac{t}{2}\nm{y-x}^2
\right\},\quad\text{ for all }x\in\R^n.
\]
As it is well-known that (see \cite[Proposition 17.2.1]{Attouch_2014} for instance) $[g_{\mathcal X}^*]_{t_k}$ is convex and continuous differentiable over $\R^n$ and $\nabla[g_{\mathcal X}^*]_{t}(x) = t(x-\prox_{g^*_{\mathcal X}/t}(x))$, we may easily conclude that $\mathcal F_k$ defined by \cref{eq:cal-Fk} is also continuous differentiable over $\R^n$. Moreover, thanks to Moreau's decomposition (cf. \cite[Theorem 6.46]{beck_first-order_2017})
\begin{equation*}
	\prox_{tg_{\mathcal X}}(x)+	t\prox_{g^*_{\mathcal X}/t}(x/t)=x
	\quad\forall\,t>0,\,x\in\R^n,
\end{equation*}
an elementary calculation gives that $\nabla \mathcal F_k(\lambda)=F_k(\lambda) $. 
Whence, from \cref{eq:Lip-Fk-ex-x-im-l}, we have $\mathcal F_k\in\mathcal S_{\theta_{k},\rho_{k}}^{1,1} $, and \cref{eq:lk1-nonlinear} 
is nothing but the Euler equation for minimizing $\mathcal F_k$.

Denote by $\partial \prox_{t_k g}^{\mathcal X}(z_k-t_kA^{\top}\lambda) $ the Clarke subdifferential \cite[Definition 2.6.1]{clarke_optimization_1987} of the monotone Lipschitz continuous mapping $\prox_{t_k g}^{\mathcal X} $ at $z_k-t_kA^{\top}\lambda$. By \cite[Chapter 7]{Facchinei2003-v2}, for all $\lambda$, it is nonempty and any $S_k(\lambda)\in\partial \prox_{t_k g}^{\mathcal X}(z_k-t_kA^{\top}\lambda) $ is positive semidefinite. If such an $S_k(\lambda)$ is symmetric, then we define an SPD operator
\begin{equation*}
	\mathcal H_k(\lambda) := \theta_{k}I+\alpha_kt_kAS_k(\lambda)A^{\top},\quad \lambda\in\R^m.
\end{equation*}
The nonsmooth version 
of Newton's method for solving \cref{eq:lk1-nonlinear} is presented as follows
\begin{equation}\label{eq:SsN}
	\lambda^{j+1} = \lambda^j - \left[\mathcal H_k(\lambda^j)\right]^{-1}F_k(\lambda^j),
	\quad j\in\mathbb N.
\end{equation}
If $\prox_{t_k g}^{\mathcal X} $ is semismooth \cite[Chapter 7]{Facchinei2003-v2}, then so is $F_k$ (see \cite[Proposition 7.4.4]{Facchinei2003-v2}) and the local superlinear convergence of the iteration \cref{eq:SsN} can be found in \cite{Qi1993,Qi1993a}. 
For global convergence, we shall perform some line search procedure \cite{dennis_numerical_1996}.

Below, in \cref{algo:SsN}, we list a semi-smooth Newton method together with a line search procedure for solving \cref{eq:lk1-nonlinear}. In practical computation, the inverse operation in \cref{eq:SsN} shall be approximated by some iterative methods. Particularly, if $S_k(\lambda)$ (and $A$) has special structure such as sparsity that allows us to do cheap matrix-vector multiplication (cf. \cite{li_highly_2017}) or construct efficient preconditioners, then one can consider PCG, as mentioned previously.
\begin{algorithm}[H]
	\caption{SsN method for solving \cref{eq:lk1-nonlinear}}
	\label{algo:SsN}
	\begin{algorithmic}[1] 
		\STATE Choose $\nu\in(0,1/2)$ and $\delta\in(0,1)$.
		\STATE Choose initial guess $\lambda\in\R^m$.
		\FOR{$j=0,1,\ldots$}
		\STATE Set $\lambda_{\rm old} = \lambda$.
		\STATE Compute $S\in\partial \prox_{t_k g}^{\mathcal X}(z_k-t_kA^{\top}\lambda_{\rm old}) $.
		\STATE Let $\mathcal H= \theta_{k}I+\alpha_kt_kASA^{\top}$ and $e = -F_k(\lambda_{\rm old})$.
		\STATE Call \cref{algo:PCG} to obtain an approximation $d$ to $\mathcal H^{-1}e$. 
		\STATE Find the smallest $r\in\mathbb N$ such that $					\mathcal F_k(\lambda_{\rm old}+\delta^rd)\leqslant \mathcal F_k(\lambda_{\rm old})+\nu\delta^r\dual{F_k(\lambda_{\rm old}),d}$.
		\STATE Update $\lambda = \lambda_{\rm old} + \delta^r d$.
		\ENDFOR
	\end{algorithmic}
\end{algorithm}
\begin{remark}	
	Note that \cref{algo:SsN} is an inexact SsN method and thus, the inner problem \cref{eq:sub-ex-x-im-l-h-g} is solved approximately. 
	Needless to say, all the methods proposed in this work have their own inner problems and for practical computation, inexact approximation shall be considered. Also, inexact convergence rate analysis shall be established but not considered in the context. \hfill\ensuremath{\blacksquare}
\end{remark}

\section{A Corrected Explicit Forward-Backward Method}
\label{sec:apd-ex-x-ex-l}
Based on \cref{eq:apd-im-x-ex-l-l,eq:apd-ex-x-im-l-x-h-g}, we consider the following scheme
\begin{subnumcases}{}
	\theta_k \frac{\lambda_{k+1}-\lambda_k}{\alpha_k} = {} \nabla_\lambda \mathcal L_\beta(v_{k+1},\lambda_{k+1}),\label{eq:apd-ex-x-ex-l-l-h-g}\\
	\frac{y_{k}-x_{k}}{\alpha_k}={}v_{k}-y_{k},\label{eq:apd-ex-x-ex-l-y-h-g}\\
	\gamma_k \frac{v_{k+1}-v_k}{\alpha_k} \in {}\mu_\beta(y_{k}- v_{k+1}) - 
	\left(\nabla h_\beta(y_k)+ \partial g_{\mathcal X}(v_{k+1})+ A^{\top}\whk\right),
	\label{eq:apd-ex-x-ex-l-v-h-g}\\
	\frac{x_{k+1}-x_{k}}{\alpha_k}={}v_{k+1}-x_{k+1},
	\label{eq:apd-ex-x-ex-l-x-h-g}
\end{subnumcases}
where $\whk$ is chosen from \cref{eq:barl-k} and the system \cref{eq:theta-gamma} is discretized via \cref{eq:apd-im-gamma-theta}. This method can be viewed as a further explicit discretization of \cref{eq:apd-ex-x-im-l-l-f}. Indeed, in step \eqref{eq:apd-ex-x-ex-l-v-h-g}, the operator splitting is still applied to $f = h+g$ but $\lambda_{k+1}$ is replaced by $\whk$. Thus $v_{k+1}$ and $\lambda_{k+1}$ are decoupled with each other, and this leads to
\begin{equation}\label{eq:apd-ex-x-ex-l-lk1-h-g-argmin}
	v_{k+1} = \mathop{\argmin}_{\mathcal X}\left\{
	g(v)+\big\langle A^{\top}\whk+	\nabla h_\beta(y_k),v\big\rangle+\frac{\tau_k}{2\alpha_k}\nm{v-w_k}^2
	\right\},
\end{equation}
where $w_{k}$ and $\tau_k$ are the same as that in \cref{eq:apd-ex-x-im-l-lk1-h-g-argmin}. Comparing \cref{eq:apd-ex-x-im-l-lk1-h-g-argmin-new,eq:apd-ex-x-ex-l-lk1-h-g-argmin}, we find the quadratic penalty term $\nm{Av-b}^2$ has been linearized.

Below, we give the convergence rate analysis of the explicit scheme \cref{eq:apd-ex-x-ex-l-x-h-g}.
\begin{thm}\label{thm:conv-apd-ex-x-ex-l-h-g}
	Assume $f=h+g$ where $h\in\mathcal S_{\mu,L}^{1,1}(\mathcal X)$ with $0\leqslant \mu\leqslant L<\infty$ and $g\in\mathcal S_0^{0}(\mathcal X)$.
	Given initial value $x_0,v_0\in \mathcal X$, the corrected explicit scheme \cref{eq:apd-ex-x-ex-l-x-h-g} generates $\{(x_k,y_k,v_k)\}\subset \mathcal X$ and if  \begin{equation}\label{eq:ak-ex-x-ex-l-h-g}
		\big(L_\beta+\nm{A}^2\big)\alpha_k^2
		=		\gamma_{k}\theta_{k} ,
	\end{equation}
	then we have the contraction
	\begin{equation}\label{eq:diff-Ek-ex-x-ex-l-h-g}
		\mathcal E_{k+1}-	\mathcal E_k
		\leqslant -\alpha_k		\mathcal E_{k+1}, \quad\text{for all~}k\in\mathbb N.
	\end{equation}
	Moreover, it holds that
	\begin{subnumcases}
		{}\nm{Ax_{k}-b}\leqslant \theta_k\mathcal R_0,\label{eq:conv-Axk-b-ex-x-ex-l}\\
		0\leqslant \mathcal L(x_{k},\lambda^*)-	\mathcal L(x^*,\lambda_{k})
		\leqslant \theta_k\mathcal E_0,
		\label{eq:conv-Lk-ex-x-ex-l}
		\\	
		{}\snm{f(x_k)-f(x^*)}
		\leqslant  \theta_k\left(\mathcal E_0+\mathcal R_0\nm{\lambda^*}\right),
		\label{eq:conv-fk-ex-x-ex-l}
	\end{subnumcases}
	where $\mathcal R_0$ is defined by \cref{eq:R0} and 
			\begin{equation}\label{eq:bk-ex-x-ex-l-correc}
		\small		\theta_k
		\leqslant
		\min\left\{
		\frac{Q}{	\sqrt{\gamma_0}k+Q}
		,\,
		\frac{Q^2}{(\sqrt{\gamma_{\min}}k+Q)^2}
		\right\}\quad\text{with}\,\,Q = 3\sqrt{L_\beta+\nm{A}^2}+\sqrt{\gamma_{\max}}.
	\end{equation}
\end{thm}
\begin{proof}
	As \cref{eq:apd-ex-x-ex-l-lk1-h-g-argmin} promises $\{v_k\}\subset \mathcal X$, it is easily concluded from \eqref{eq:apd-ex-x-ex-l-y-h-g} and \eqref{eq:apd-ex-x-ex-l-x-h-g} that $\{(x_k,y_k)\}\subset \mathcal X$ as long as $x_0,v_0\in \mathcal X$.
	
	The proof of \cref{eq:diff-Ek-ex-x-ex-l-h-g} is almost in line with that of \cref{eq:diff-Ek-ex-x-im-l-correc}. The identity \cref{eq:I1-ex-x-im-l-h-g} of the first term $I_1$ leaves unchanged here. For $I_2$, we mention the estimate \cref{eq:I2-im-x-ex-l}:
	\[
	I_2\leqslant -\frac{\alpha_k\theta_{k+1}}{2}\nm{\lambda_{k+1}-\lambda^*}^2
	+\frac{\theta_{k}}{2}\big\|\lambda_{k+1}-\whk\big\|^2
	+\alpha_{k}\big\langle Av_{k+1}-b,\whk-\lambda^*\big\rangle.
	\]
	The expansion of $I_3$ is tedious but the same as what we did in the proof of \cref{thm:apd-ex-x-im-l-h-g}, with $\lambda_{k+1}$ being $\whk$. For simplicity, we will not go through the details here once again. Consequently, one observes that \cref{eq:diff-ex-x-im-l-h-g} now turns into
	\[
	\begin{aligned}
		\mathcal E_{k+1}-\mathcal E_{k}\leqslant &
		-\alpha_k\mathcal E_{k+1}
		-	\frac{\gamma_{k}}2\nm{v_{k+1}-v_k}^2
		+\frac{\theta_{k}}{2}\big\|\lambda_{k+1}-\whk\big\|^2\\
		{}&	\quad	+(1+\alpha_k)\left(h_\beta(x_{k+1})-h_\beta(y_{k})\right)	-\alpha_k\dual{\nabla h_\beta(y_k),v_{k+1} - v_k},
	\end{aligned}
	\]
	where the last line in terms of $g$ is nonpositive and has been dropped sine $x_{k+1}$ is a convex combination of $x_k$ and $v_{k+1}$.
	Noticing that the relation $(1+\alpha_k)	(x_{k+1}-y_k) = \alpha_k(v_{k+1}-v_k)$ holds true for \cref{eq:apd-ex-x-ex-l-l-h-g}, we still have the estimate \cref{eq:fk1-fk-ex-x-im-l} here. This implies 
	\[
	\begin{aligned}
		\mathcal E_{k+1}-\mathcal E_{k}\leqslant &
		-\alpha_k\mathcal E_{k+1}	+\frac{\theta_{k}}{2}\big\|\lambda_{k+1}-\whk\big\|^2
		+\left(\frac{L_\beta\alpha_k^2}{2(1+\alpha_k)}
		-\frac{\gamma_k}{2}\right)
		\nm{v_{k+1}- v_k}^2.
	\end{aligned}
	\]
	By \cref{eq:barl-k} and \eqref{eq:apd-ex-x-ex-l-l-h-g}, we have $\lambda_{k+1}-	\whk =\alpha_k/\theta_k A(v_{k+1}-v_k)$ and it follows that
	\begin{equation}\label{eq:last-est-ex-ex}
		\begin{aligned}
			\mathcal E_{k+1}-\mathcal E_{k}\leqslant {}&
			-\alpha_k\mathcal E_{k+1}
			+\frac{1}{2\theta_{k}}
			\big(\alpha_k^2(L_\beta\theta_{k+1}
			+\nm{A}^2)-\gamma_k\theta_k\big)
			\nm{v_{k+1}-v_k}^2.
		\end{aligned}
	\end{equation}
	Thanks to \cref{eq:ak-ex-x-ex-l-h-g} and the evident fact $\theta_{k+1}\leqslant \theta_0=1$, the last term is nonpositive, which proves \cref{eq:diff-Ek-ex-x-ex-l-h-g}.
	
	Proceeding as before, it is not hard to establish \cref{eq:conv-fk-ex-x-ex-l}. 
	As the decay estimate \cref{eq:bk-ex-x-ex-l-correc} is similar with \cref{eq:bk-est-im-x-ex-l}, 
	we conclude the proof of this theorem.
\end{proof}
\begin{remark}
	From the estimate \cref{eq:last-est-ex-ex}, one may observe the fancy choice 
	\[
	\big(	L_\beta\theta_{k+1}+\nm{A}^2\big)\alpha_k^2=\gamma_k\theta_k.
	\]
	This gives an algebraic equation in terms of $\alpha_k$ with degree three because $\theta_{k+1} = \theta_{k}/(1+\alpha_k)$. It is not a problem to determine $\{\alpha_k\}$ but such a sequence does not improve the asymptotic decay rate of $\{\theta_k\}$, as given in \cref{eq:bk-ex-x-ex-l-correc}. Hence, we chose a more simple one \cref{eq:ak-ex-x-ex-l-h-g}.
\hfill\ensuremath{\blacksquare}
\end{remark}

Now let us summarize \cref{eq:apd-ex-x-ex-l-x-h-g} together with the step size \cref{eq:ak-ex-x-ex-l-h-g} in  \cref{algo:Ex-APDFB-h-g}, which is called the explicit accelerated primal-dual forward-backward (Ex-APDFB) method. 
\begin{algorithm}[H]
	\caption{Ex-APDFB method for $f = h+g$ with $h\in\mathcal S_{\mu,L}^{1,1}(\mathcal X)$ and $g\in\mathcal S_0^0(\mathcal X)$}
	\label{algo:Ex-APDFB-h-g}
	\begin{algorithmic}[1] 
		\REQUIRE  $\beta\geqslant 0,\,\theta_0=1,\,\gamma_0>0,\, (x_0,v_0)\in\mathcal X\times \mathcal X,\,\lambda_0\in\R^m$.
		\STATE Set $\beta=0$ if $\sigma_{\min}(A)=0$, and let $\mu_\beta=\mu+\beta\sigma_{\min}^2(A)$.
		\STATE Set $L_\beta=L+\beta\nm{A}^2$ and $S_\beta=L_\beta+\nm{A}^2$.
		\FOR{$k=0,1,\ldots$}
		\STATE Choose step size $\alpha_k=\sqrt{\theta_{k}\gamma_k/S_\beta}$.
		\STATE Update $\displaystyle \gamma_{k+1} = (\gamma_k+\mu_\beta\alpha_k)/(1+\alpha_k)$ and $\displaystyle \theta_{k+1}= \theta_k/(1+\alpha_k)$.
		\STATE Set $\displaystyle  \tau_k = \gamma_k+\mu_\beta\alpha_k,\,\eta_k=\alpha_k/\tau_k$ and $y_{k}={}(x_k+\alpha_kv_{k})/(1+\alpha_k)$.
		\STATE Set $\displaystyle  	w_{k}={}\tau_k^{-1}(\gamma_kv_{k}+\mu_\beta\alpha_ky_k)$ and $\whk
		=\lambda_k+\alpha_k/\theta_k\left(
		Av_k-b		\right)$.
		\STATE Update $\displaystyle	v_{k+1}
		=\prox_{\eta_kg}^{\mathcal X}(w_k-\eta_k(\nabla h_\beta(y_k)+A^{\top}\whk))
		$.
		\STATE Update $\displaystyle  x_{k+1}={}(x_k+\alpha_kv_{k+1})/(1+\alpha_k)$.
		\STATE Update $\lambda_{k+1}=
		\lambda_k+\alpha_k/\theta_k\left(
		Av_{k+1}-b
		\right)$.
		\ENDFOR
	\end{algorithmic}
\end{algorithm}
To the end, we mention some comparisons with related works.
In view of the estimate \cref{eq:bk-ex-x-ex-l-correc}, we have
\begin{equation}\label{eq:rate-ex-apd-fb}
	{}\snm{f(x_k)-f(x^*)}	+ \nm{Ax_k-b}	\leqslant 
	C
	\left\{
	\begin{aligned}
		&		\frac{\nm{A}+\sqrt{L}}{k},&&\mu_\beta = 0,\\
		&	\frac{\nm{A}^2+L}{k^2},&&\mu_\beta > 0.
	\end{aligned}
	\right.
\end{equation}
This may give a negative answer to the question addressed in the conclusion part of  \cite{Xu2017}. That is, can we linearize the augmented term and maintain the nonergodic rate $O(1/k^2)$ 
under the assumption that $f=h+g$ is convex and $h$ has $L$-Lipschitz continuous gradient? According to \cref{eq:rate-ex-apd-fb}, if $\mu_\beta>0$, which means either  $\mu>0$ or $\sigma_{\min}(A)>0$ (i.e., $A$ has full column rank), then the rate $O(1/k^2)$ is maintained. Otherwise, it slows down to $O(1/k)$. We also notice that, for strongly convex case, the rate $O(1/k^2)$ of the fully linearized proximal ALM in \cite{Xu2017} is in ergodic sense.

As mentioned at the end of \cref{sec:apd-im}, the sequence $\{(x_k,y_k,v_k,\lambda_k)\}$ in \cref{algo:Ex-APDFB-h-g} can be further simplified to $\{(y_k,v_k)\}$ or $\{(x_k,v_k)\}$ if we drop $\{\lambda_k\}$, by using \eqref{eq:apd-ex-x-ex-l-l-h-g}.
When $\mathcal X = \R^n$, \cref{algo:Ex-APDFB-h-g} is very close to the accelerated penalty method in \cite{li_convergence_2017}, which also produces some two-term sequence $\{(x_k,y_k)\}$. Moreover, they share the same nonergodic convergence rate (cf. \cref{eq:rate-ex-apd-fb} and \cite[Theorem 4]{li_convergence_2017}).

\section{Application to  Decentralized Distributed Optimization}
\label{sec:ddo}
In this part, we focus on numerical performance of \cref{algo:Semi-APDFB-h-g} for solving decentralized distributed optimization.

Assume there is some simple connected graph $G = (V,E)$ with $n = |V|$ nodes. Each node $i\in V$ stands for an agent who accesses the information of a smooth convex objective $f_i:\R^m\to\R$ and communicates with its neighbor $N(i): = \{j\in V:(i,j)\in E\}$. The goal is to minimize the average
\begin{equation}\label{eq:min-fi-ddo}
	\min_{x\in\R^m} 
	\frac{1}{n}\sum_{i=1}^{n}f_i(x).
\end{equation}
Let $q=mn$ and introduce a vector ${\bm x}\in \R^{q}$ which has $n$ blocks. 
Each block ${\bm x}(i)\in\R^m$ is located at node $i$ and becomes a local variable 
with respect to $f_i$. Then, \cref{eq:min-fi-ddo} can be reformulated as follows
\begin{equation}\label{eq:min-f-ddo}
	\min_{\bm x\in\R^q} {f}(\bm x): = \frac{1}{n}\sum_{i=1}^{n}f_i(\bm x(i)),
	\quad \st\, \bm x(1) = \cdots=\bm x(n).
\end{equation}
We mainly consider the smooth convex case $f_i\in\mathcal S_{\mu_i,L_i}^{1,1}$ with $0\leqslant \mu_i\leqslant L_i<\infty$, which implies that ${f}\in\mathcal S_{\mu,L}^{1,1}$ with ${\mu}= \min\{\mu_i\}/n$ and ${L}= \max\{L_i\}/n$.

As we see from \cref{eq:min-f-ddo}, there comes an additional constraint, called the consensus restriction. One popular way to treat this condition is to introduce some matrix ${A}\in\R^{q\times q}$ that is symmetric positive semi-definite with null space ${\rm span }\{{\bf 1}_{q}\}$, where ${\bf 1}_q\in\R^{q}$ denotes the vector of all ones. Then \cref{eq:min-f-ddo} can be rewritten as the same form of \cref{eq:min-f-X-Ax-b}:
\begin{equation}\label{eq:min-f-A-ddo}
	\min_{\bm x\in\R^q} f(\bm x) 
	\quad \st A\bm x  = 0,
\end{equation}
which is also equivalent to 
\begin{equation}\label{eq:min-f-sq-A-ddo}
	\min_{\bm x\in\R^q} f(\bm x) 
	\quad \st \sqrt{A}\bm x  = 0.
\end{equation}
Indeed, we have $\sqrt{A}\bm x = 0\Longleftrightarrow A\bm x = 0$ since $A$ is positive semi-definite. Besides, as the null space of $A$ is ${\rm span}\{{\bf 1}_q\}$, it follows that $A\bm x=0\Longleftrightarrow\bm x(1) = \cdots=\bm x(n)$.

There are many candidates for the matrix $A$. Here we adopt $A = \Delta_G\otimes I_m$, where 
$I_m$ is the identity matrix of order $m$ and $\Delta_G = D_G-A_G$ is the Laplacian matrix of the graph $G$, with $D_G$ being the diagonal matrix of vertex degree and $A_G$ being the adjacency matrix of $G$. As $G$ is connected, by \cite[Lemma 4.3]{Bapat2014}, the null space of $\Delta_G$ is 
${\rm span}\{{\bf 1}_n\}$. This means the current $A$ satisfies our demand.

To solve \cref{eq:min-fi-ddo}, we apply \cref{algo:Semi-APDFB-h-g} to problem \cref{eq:min-f-sq-A-ddo} and further simplify it as \cref{algo:Semi-APD-DDO}, where we set $\beta=0$ since $\sigma_{\min}(\sqrt{A}) = 0$ and choose $\lambda_0=\sqrt{A}{\bm x}_0 $ to eliminate $\{\lambda_{k}\}$ since by \eqref{eq:apd-ex-x-im-l-l-h-g} and \eqref{eq:apd-ex-x-im-l-x-h-g} we have that
\[
\lambda_{k+1} - \theta_{k+1}^{-1}\sqrt{A}{\bm x}_{k+1}
=\lambda_{k} - \theta_{k}^{-1}\sqrt{A}{\bm x}_k=\cdots
=\lambda_{0} - \sqrt{A}{\bm x}_0 =0,
\]
which implies $\lambda_k = \theta_k^{-1}\sqrt{A}{\bm x}_k$ for all $k\in\mathbb N$. Recall that for \cref{eq:min-f-sq-A-ddo} the key step is to compute ${\bm v}_{k+1}$ from \cref{eq:vk1-SPD}, which now reads as follows
\begin{equation}\label{eq:vk1}
	(\epsilon_kI+A){\bm v}_{k+1}	={\bm s}_k,
\end{equation}
where $\epsilon_k=\tau_k\theta_k/\alpha_k^{2}$ and ${\bm s}_k=\epsilon_k{\bm z}_k-A{\bm x}_k/\alpha_k$. 
Since $\gamma_{k}=L\alpha_k^2$ and $\tau_k=\gamma_k+\mu\alpha_k$, we have $\epsilon_k = O(\theta_k)$. Therefore, \cref{eq:vk1} is a nearly singular SPD system and careful iterative 
method shall be considered. Instead of solving the original system \cref{eq:vk1}, in the next part, we shall discuss how to obtain ${\bm v}_{k+1}$ efficiently, by applying  PCG iteration (i.e., \cref{algo:PCG}) to the augmented system \cref{eq:aug-vk1}.
\begin{algorithm}
	\caption{Semi-APDFB method for \cref{eq:min-f-sq-A-ddo} with $f\in\mathcal S_{\mu,L}^{1,1},\,0\leqslant \mu\leqslant L<\infty$}
	\label{algo:Semi-APD-DDO}
	\begin{algorithmic}[1] 
		\REQUIRE  $\gamma_0,\, \bm x_0,\,{\bm v}_0\in\R^{q}$.
		\FOR{$k=0,1,\ldots$}
		\STATE Choose step size $\alpha_k=\sqrt{\gamma_k/L}$.
		\STATE Update $\displaystyle \gamma_{k+1} = (\gamma_k+\mu\alpha_k)/(1+\alpha_k)$ and $\displaystyle \theta_{k+1}= \theta_k/(1+\alpha_k)$.
		\STATE Set $\displaystyle  \tau_k = \gamma_k+\mu\alpha_k$ and $\displaystyle  {\bm y}_{k}={}(\bm x_k+\alpha_k{\bm v}_{k})/(1+\alpha_k)$.
		\STATE Set  $		{\bm w}_{k}={}\tau_k^{-1}(\gamma_k{\bm v}_{k}+\mu\alpha_k{\bm y}_k)$ and ${\bm z}_k = {\bm w}_k-\alpha_k/\tau_k\nabla f({\bm y}_k)$.
		\STATE Solve $\widehat{\bm v}=(\widehat{ v}_1,\widehat{\bm v}^{\top}_2)^{\top}$ from \cref{eq:aug-vk1} via \cref{algo:PCG} with Jacobi preconditioner and the tolerance $\varepsilon=\|A{\bm x}_k\|/10$.
		\STATE Recover ${\bm v}_{k+1} = \widehat{ v}_1{\bf 1}_q+\widehat{\bm v}_2$.		
		\STATE Update $\displaystyle  \bm x_{k+1}={}(\bm x_k+\alpha_k{\bm v}_{k+1})/(1+\alpha_k)$.
		\ENDFOR
	\end{algorithmic}
\end{algorithm}
\subsection{Robust null space method for \cref{eq:vk1}}
For simplicity, let us fix $k$ and write $\epsilon=\epsilon_k,\,{\bm s} = {\bm s_k}$ and $A_\epsilon =\epsilon I+A $.
Note that the condition number of $A_\epsilon$ is $1+\lambda_{\max}(A)/\epsilon$. 
Therefore, classical iterative methods, such as Jacobi and Gauss-Seidel (GS) iterations, have to converge dramatically slowly as $\epsilon$ becomes small.

Recall that the null space of $A$ is ${\rm span}\{{\bf 1}_{q}\}$. 
Following \cite{lee_robust_2007,padiy_generalized_2001}, let us 
introduce the {\it augmented system} of \cref{eq:vk1} by that
\begin{equation}\label{eq:aug-vk1}
	\mathcal A\widehat{\bm v} =\begin{pmatrix}
		\epsilon  q&		\epsilon{\bm 1}_q^{\top}		\\
		\epsilon{\bf 1}_q&		A_\epsilon\\
	\end{pmatrix}
	\begin{pmatrix}
		\widehat{ v}_1\\\widehat{\bm v}_2
	\end{pmatrix}
	=
	\begin{pmatrix}
		{\bf 1}_q^{\top}{\bm s}\\ 	{\bm s}
	\end{pmatrix}
	= \widehat{\bm s}.
\end{equation}
Clearly, this system is singular  and has infinitely many solutions but the solution $\bm v$ to \cref{eq:vk1} can be uniquely recovered from ${\bm v} = \widehat{ v}_1{\bf 1}_q+\widehat{\bm v}_2$, where $\widehat{\bm v}=(\widehat{ v}_1,\widehat{\bm v}_2^{\top})^{\top}$ is any solution to the augmented system \cref{eq:aug-vk1}. 

The Jacobi method for \cref{eq:aug-vk1}, which is also a block iteration since $A = \Delta_G\otimes I_m$, reads as follows: given the $l$-th iteration $\widehat{\bm v}^l=(\widehat{ v}^l_1,\widehat{\bm v}^l_2{}^\top)^{\top}$, do the next one:  
\begin{equation}\label{eq:J}
	\widehat{ v}_{1}^{l+1} = {}\frac{{\bf 1}_q^{\top}\widehat{\bm v}_{2}^{l}}{q}	,\quad
	\widehat{\bm v}_{2}^{l+1}(i)={} 	\frac{1}{\epsilon+a_{ii}}\left({\bm s}(i)-\epsilon\widehat{ v}^{l}_1+\sum_{\substack{j\in N(i)}}\widehat{\bm v}_{2}^{l}(j)\right),
\end{equation}
simultaneously for $1\leqslant i\leqslant n$. The GS iteration for \cref{eq:aug-vk1}, which is also a block GS method, is formulated as follows
\begin{equation}\label{eq:GS}
	\small
	\widehat{ v}_{1}^{l+1} = \frac{{\bf 1}_q^{\top}\widehat{\bm v}_{2}^{l}}{q}	,\quad
	\widehat{\bm v}_{2}^{l+1}(i)={} 	\frac{1}{\epsilon+a_{ii}}\left({\bm s}(i)-\epsilon\widehat{ v}^{l+1}_1+\sum_{\substack{j<i\\j\in N(i)}}\widehat{\bm v}_{2}^{l+1}(j)+\sum_{\substack{j>i\\j\in N(i)}}\widehat{\bm v}_{2}^{l}(j)\right),
\end{equation}
sequentially for $1\leqslant i\leqslant n$.
One can also consider the symmetrized version, i.e., the symmetry Gauss-Seidel (SGS) method \cite{xu_method_2001}.

In \cite[Lemma 3.1]{lee_robust_2007}, it has been analyzed that the GS iteration \cref{eq:GS} for the augmented system \cref{eq:aug-vk1} is robust in terms of $\epsilon$, and when $\epsilon\to0$, the convergence rate converges to that 
of the the GS iteration for the singular system $A{\bm v} = {\bm s}$ (with ${\bm s}$ belonging to the range of $A$). As further proved in \cite[Theorem 4.1]{lee_robust_2007}, the iteration \cref{eq:GS} is nothing but a successive subspace correction method for \cref{eq:vk1} with respect to a special space decomposition $\R^{q} = {\rm span}\{{\bf 1}_q\}+\sum_{i=1}^{q}{\rm span}\{{\bm e}_i\}$, where ${\bm e}_i$ is the $i$-th canonical basis of $\R^q$. Recall that ${\rm span}\{{\bf 1}_q\}$ happens to be the null space of $A$.

For concrete illustration, we generate two simple connected graphs from the package DistMesh (cf. \cite{persson_simple_2004} or \url{http://persson.berkeley.edu/distmesh/}); see \cref{fig:sphere_torus}. They are surface meshes on the unit sphere and a torus, respectively. The former has 480 nodes and 1434 edges, and the latter possesses 640 nodes and 1920 edges. They share the same average vertex degree 6.
\begin{figure}[H]
	\centering
	\includegraphics[width=320pt,height=150pt]{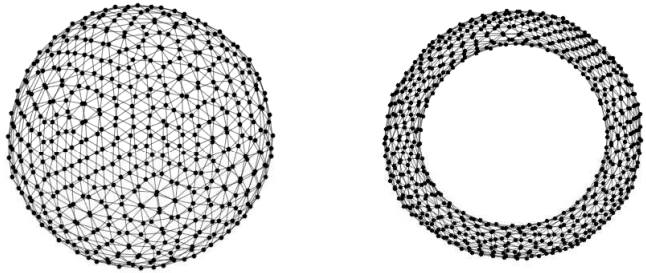}
	\caption{Two connected graphs on the surfaces of the unit sphere (left) and a torus (right). The left has 480 nodes and 1434 edges, and the right has 640 nodes and 1920 edges. The average vertex degree is 6.}
	\label{fig:sphere_torus}
\end{figure}

For simplicity, we consider $m=1$ which means $A=\Delta_G$ is the Laplacian of  the graph $G$. Performances of Jacobi, GS and SGS iterations for the original SPD system \cref{eq:vk1} and the augmented system \cref{eq:aug-vk1} are reported in \cref{tab:robust-test_sphere,tab:robust-test_torus}. Also, results of PCG (i.e., \cref{algo:PCG}) with Jacobi and SGS preconditioners for the augmented system \cref{eq:aug-vk1} are given. All the iterations are stopped either the maximal iteration number 1e5 is attained or the relative residual is smaller than 1e-6.
\renewcommand\arraystretch{1}
\begin{table}[H]
	\centering
	\caption{Performances of iterative solvers for \cref{eq:vk1,eq:aug-vk1}, related to the graph on the unit sphere in \cref{fig:sphere_torus}. Here, $\times$ means the maximal iteration number 1e5 is attained while the relative residual is larger than 1e-6.}
	\label{tab:robust-test_sphere}
	\small\setlength{\tabcolsep}{1.2pt}
	\begin{tabular}{cccccccccccccccccc}
		\toprule
		&\phantom{a} &\multicolumn{5}{c}{\cref{eq:vk1}}&
		&\phantom{a} &  \multicolumn{9}{c}{\cref{eq:aug-vk1}}\\
		\cmidrule{3-7} 	\cmidrule{10-18} 
		$\epsilon$&&Jacobi& &GS&    & SGS   &&&Jacobi& &GS&    & SGS   && PCG-Jacobi&    & PCG-SGS \\
		\midrule
		$1e$-1&&746&&	375	&&212&&&	232&&	113&&	66&&	30	&&16\\
		$1e$-2&&7439	&&3740	&&2104&&&	427&&	214&&	117&&39&&	19\\
		$1e$-3&&74358	&&37463	&&20968&&&	461&&	237&&	134&&	41&&	19\\
		$1e$-4&&$\times$&&	$\times$&&	$\times$	&&&479&&	241&&	131&&	41	&&19\\
		$1e$-6&&$\times$&&	$\times$&&	$\times$&&&	468	&&236	&&131&&	41&&	19\\
		$1e$-8&&$\times$&&	$\times$&&	$\times$	&&&468&&	233&&	131&&	41&&	19\\
		\bf 0&&\bf 473	&&\bf 235&&	\bf134&&&	-&&	-&&	-&&	-&&	-\\
		\bottomrule
	\end{tabular}
\end{table}

It is observed that all the iterations for the augmented system \cref{eq:aug-vk1} are 
robust with respect to $\epsilon$, and PCG with SGS preconditioner performances the best. However, we have to mention that, in the setting of decentralized distributed optimization, both GS and SGS iterations may not be preferable since all the nodes are updated {\it sequentially}. The Jacobi iteration \cref{eq:J} is parallel but another issue, which also exists in the GS iteration \cref{eq:GS}, is that there comes an additional variable $	\widehat{ v}_{1}\in \R$, which is updated via the average of $\widehat{\bm v}_{2}$. Moreover, to recover ${\bm v}_{k+1}=\widehat{ v}_{1}{\bf 1}_q+\widehat{\bm v}_{2}$, all nodes need it. This can be done by introducing a master node that connects all other nodes and is responsible for updating $\widehat{ v}_{1}$ and then sending it back to local nodes. Note again that both $\widehat{ v}_{1}$ and $\widehat{\bm v}_{2}$ can be obtained simultaneously for Jacobi iteration. Therefore, the master and other nodes are allowed to be asynchronized. This maintains the decentralized nature of distributed optimization. 
\begin{table}[H]
	\centering
	\caption{Performances of iterative solvers for \cref{eq:vk1,eq:aug-vk1}, related to the graph on the torus in \cref{fig:sphere_torus}. Here, $\times$ means the maximal iteration number 1e5 is attained while the relative residual is larger than 1e-6.}
	\label{tab:robust-test_torus}
	\small\setlength{\tabcolsep}{1.2pt}
	\begin{tabular}{cccccccccccccccccc}
		\toprule
		&\phantom{a} &\multicolumn{5}{c}{\cref{eq:vk1}}&
		&\phantom{a} &  \multicolumn{9}{c}{\cref{eq:aug-vk1}}\\
		\cmidrule{3-7} 	\cmidrule{10-18} 
		$\epsilon$&&Jacobi& &GS&    & SGS   &&&Jacobi& &GS&    & SGS   && PCG-Jacobi&    & PCG-SGS \\
		\midrule	
		$1e$-1&&751&&	378&&	216&&&	319&&	167&&	92&&	35&&	17\\
		$1e$-2&&7463&&	3759&&	2147&&&	1031&&	526&&	294&&	57&&	22\\
		$1e$-3&&74825	&&37673&&	21408	&&&1356&&	684&&	384&&	59&&	27\\
		$1e$-5&&$\times$	&&$\times$&&	$\times$&&&	1396&&	708&&	399&&	60&&	27\\
		$1e$-7&&	$\times$	&&$\times$&&	$\times$&&&	1397&&	707	&&396&&	59&&	27\\
		$1e$-9&&	$\times$&&	$\times$&&	$\times$&&&	1396&&	701&&	400	&&60	&&27\\
		\bf 0&&\bf1156&&\bf	634&&	\bf318&&&	-&&	-&&	-&&	-&&	-\\
		\bottomrule
	\end{tabular}
\end{table}
We have not presented convergence analysis in inexact setting but for all the forthcoming numerical tests in \cref{sec:ddo-lso,sec:ddo-log}, we adopt \cref{algo:PCG} with Jacobi preconditioner and the tolerance $\varepsilon=\|A{\bm x}_k\|/10$ to solve the augmented system \cref{eq:aug-vk1}; see step 6 in \cref{algo:Semi-APD-DDO}. 
\subsection{Decentralized least squares}
\label{sec:ddo-lso}
Let us now consider the decentralized least squares
\begin{equation}\label{eq:ddo-lso}
	\min_{x\in\R^m} 
	\frac{1}{n}\sum_{i=1}^{n}f_i(x)= 
	\frac{1}{n}\sum_{i=1}^{n}\frac{1}{2}\nm{B_ix-b_i}^2,
\end{equation}
where $B_i\in\R^{p\times m}$ and $b_i\in\R^{p}$ are randomly generated at each node $i$. 
Here we set $m = 200$ and the sample number $p= 5$. Note that each $f_i$ in \cref{eq:ddo-lso} is smooth convex with $\mu_i=0$ and $L_i = \nm{B_i}^2$, and for $f(\bm x) = \frac{1}{n}\sum_{i=1}^{n}f_i(\bm x(i))$, we have $f\in\mathcal S_{0,L}^{1,1}$ with $L = \max\{L_i\}/n$.
\begin{figure}
	\centering
	\includegraphics[width=320pt,height=300pt]{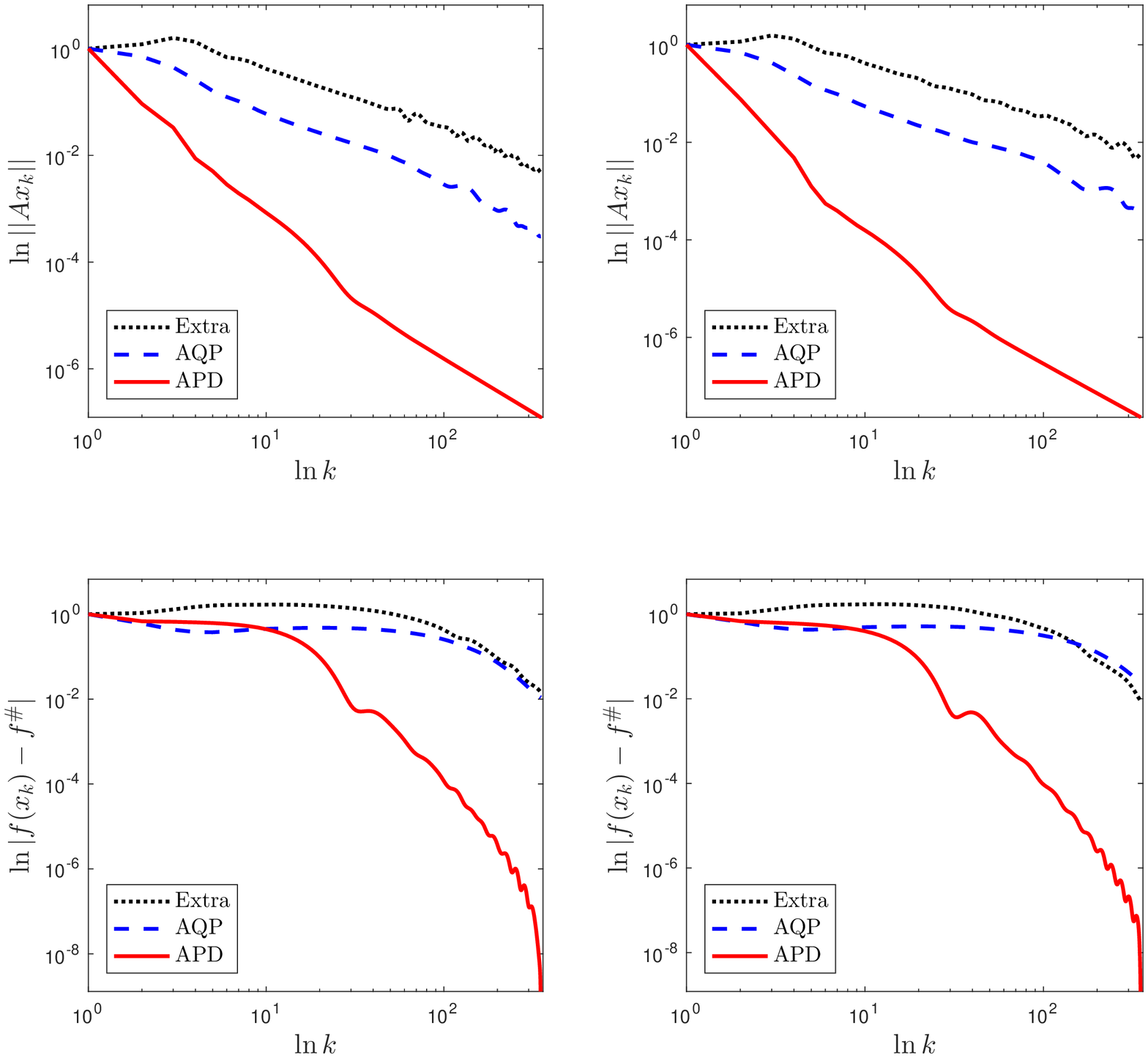}
	\caption{Convergence behaviors of \cref{algo:Semi-APD-DDO}, Extra and AQP for problem \cref{eq:ddo-lso} on the sphere graph (left) and the torus graph (right). Here $f^{\#}$ is the approximated optimal objective value.}
	\label{fig:ddo-lsr}
\end{figure}

We compare \cref{algo:Semi-APD-DDO} with Extra \cite{Shi2015} 
and the accelerated quadratic penalty (AQP)  method \cite{li_convergence_2017} for solving \cref{eq:ddo-lso} 
with respect to the previous two connected graphs (plotted in \cref{fig:sphere_torus}). 
Starting from the problem \cref{eq:min-f-ddo}, Extra requires the so-called {\it mixing matrix} $W$, which is related to the underlying graph $G$ and satisfies \cite[Assumption 1]{Shi2015}, and it repeats the  iteration procedure below
\begin{equation}\label{eq:Extra}
	\left\{
	\begin{aligned}
		{}&		{\bm e}_{k}=\bm x_{k}-\left(\widehat{W}\otimes I_m\right)\bm x_{k-1}+\alpha\nabla f(\bm x_{k-1}),\\		
		{}&		\bm x_{k+1} =\left(W\otimes I_m\right)\bm x_{k}-\alpha\nabla f(\bm x_{k})+{\bm e}_{k},
	\end{aligned}
	\right.
\end{equation}
for $k\geqslant 1$, where $\widehat{W}= (I+W)/2$ and the initial step is $		\bm x_1 = {}\left(W\otimes I_m\right)\bm x_0-\alpha\nabla f(\bm x_0)$. Assuming the spectrum of $W$ lies in $(-1,1]$ and that $\alpha=\lambda_{\min}(\widehat{W})/L$, \cite[Theorem 3.5]{Shi2015} gave the ergodic sublinear rate $O(1/k)$ for \cref{eq:Extra}. The AQP method \cite[Eq.(9)]{li_convergence_2017} rewrites \cref{eq:ddo-lso} as the form \cref{eq:min-f-sq-A-ddo} with $A = (I-U)/2\otimes I_m$ and performs the following iteration 
\begin{equation}\label{eq:AQP}
	\left\{
	\begin{aligned}
		{}&{\bm y}_k = {\bm x}_k+\frac{k-1}{k+1}(\bm x_k-\bm x_{k-1}),\\
		{}&		\bm x_{k+1} = {\bm y}_k - \frac{\nabla f({\bm y}_k)+(k+1)A{\bm y}_k}{L+k+1},
	\end{aligned}
	\right.
\end{equation}
for all $k\geqslant 1$, where $U$ is some symmetric doubly stochastic matrix such that $U_{ij}>0$ if and only of $(i,j)\in E$. The nonergodic convergence rate $O(1/k)$ for \cref{eq:AQP} has been established in \cite[Theorem 6]{li_convergence_2017}.

In this example and the next one, we choose 
$W =U= I-\Delta_G/\tau$ with $\tau=\lambda_{\max}(\Delta_G)$. Then $W$ fulfills \cite[Assumption 1]{Shi2015} and by \cite[Theorem 4.12]{Bapat2014}, such $U$ also meets the requirement in \cref{eq:AQP}. In \cref{fig:ddo-lsr}, we plot the convergence behaviors of Extra, AQP and APD (\cref{algo:Semi-APD-DDO}). By \cref{thm:apd-ex-x-im-l-h-g}, APD converges with a faster sublinear rate $O(1/k^2)$ and numerical results  illustrate that our method outperforms the others indeed.

\subsection{Decentralized logistic regression}
\label{sec:ddo-log}
We then look at the regularized decentralized logistic regression
\begin{equation}\label{eq:ddo-lr}
	\min_{x\in\R^m}
	\frac{1}{n}\sum_{i=1}^{n}f_i(x)=
	\frac{1}{n}\sum_{i=1}^{n}\left(\ln\left(1+\exp\left(-b_i\theta_i^{\top}x\right)\right)
	+\frac{\delta}{2}\nm{x}^2\right),
\end{equation}
where $\delta>0$ stands for the regularize parameter, $\theta_i\in\R^{ m}$ is the data variable and $b_i\in\{-1,1\}$ denotes the binary class. Here we take $\delta=0.5$ and $m = 300$. Note that each $f_i$ is smooth 
strongly convex and an elementary computation gives $\mu_i=\delta$ and $ L_i = \delta+|b_i|^2\nm{\theta_i}^2/4$.  Hence $f(\bm x) = \frac{1}{n}\sum_{i=1}^{n}f_i(\bm x(i))$ is also smooth strongly convex with $\mu= \delta/n$ and $L = \max\{L_i\}/n$.
\begin{figure}[H]
	\centering
		\includegraphics[width=320pt,height=300pt]{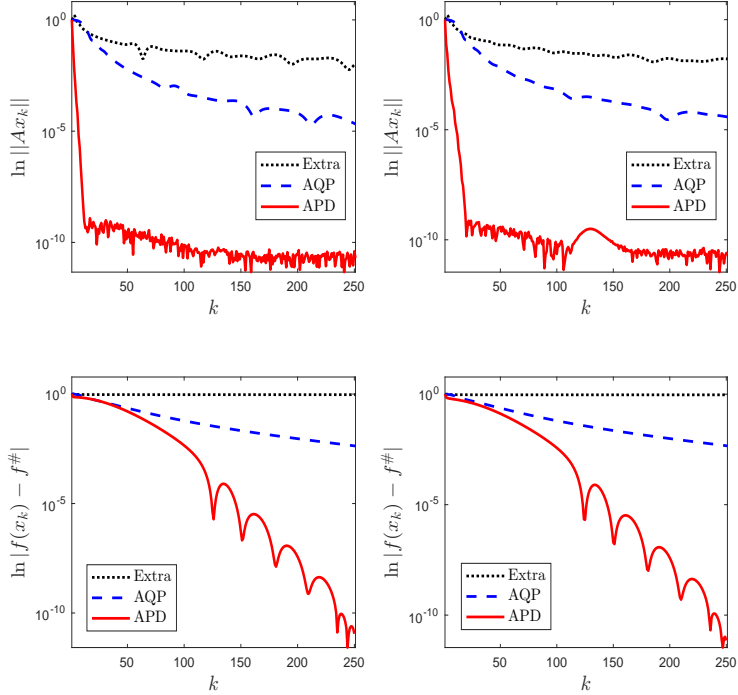}
	\caption{Convergence behaviors of \cref{algo:Semi-APD-DDO}, Extra and AQP for problem \cref{eq:ddo-lr} on the sphere graph (left) and the torus graph (right). Here $f^{\#}$ is the approximated optimal objective value.}
	\label{fig:ddo-log}
\end{figure}

In this case, the corresponding variant of AQP \cref{eq:AQP} has the theoretical sublinear rate $O(1/k^2)$ and reads as follows
\begin{equation}\label{eq:aqp-mu}
	\left\{
	\begin{aligned}
		{}&{\bm y}_k = {\bm x}_k+\frac{(\eta_k\theta_k-\mu\theta^2_k)(1-\theta_{k-1})}{(\eta_k-\mu\theta^2_k)\theta_{k-1}}(\bm x_k-\bm x_{k-1}),\\
		{}&		\bm x_{k+1} = {\bm y}_k -\eta_k^{-1}\left(\theta^2_k\nabla f({\bm y}_k)+\mu A{\bm y}_k\right),
	\end{aligned}
	\right.
\end{equation}
where $\eta_k = L\theta_k^2+\mu$ and $\theta_k^2+\theta_{k-1}^2\theta_k =\theta_{k-1}^2$ with $\theta_0=1$. By \cite[Theorem 3.7]{Shi2015}, Extra \cref{eq:Extra} has linear convergence with the step size $\alpha=\mu\lambda_{\min}(\widehat{W})/L^2$. However, numerical outputs in \cref{fig:ddo-log} show it performances even worse than AQP \cref{eq:aqp-mu}. This may be due to that the mixing matrix $W$ in \cref{eq:Extra} is not chosen properly and not much efficient for information diffusion in the graph. There are some alternative choices summarized in \cite[Section 2.4]{Shi2015} and we tried the Metropolis constant edge weight matrix, which performs not much better either and is not displayed here. We would not look at more mixing matrices beyond. To conclude, we observe fast linear convergence of APD (\cref{algo:Semi-APD-DDO}) from \cref{fig:ddo-log}, for both the objective gap and the feasibility.

\section{Concluding Remarks}
\label{sec:apd-conclu}
In this work, for minimizing a convex objective with linear equality constraint, we introduced a novel second-order dynamical system, called accelerated primal-dual flow, and proved its exponential decay property in terms of a suitable Lyapunov function. It was then discretized via different type of numerical schemes, which give a class of accelerated primal-dual algorithms for the affine constrained convex optimization problem \cref{eq:min-f-X-Ax-b}. 

The explicit scheme \cref{eq:apd-ex-x-ex-l-x-h-g} corresponds to fully linearized proximal ALM and semi-implicit discretizations (cf. \cref{eq:apd-im-x-ex-l-v} and \cref{eq:apd-ex-x-im-l-x-h-g}) are close to partially linearized ALM. The subproblem of \cref{eq:apd-ex-x-im-l-x-h-g} has special structure and can be used to develop efficient inner solvers. Also, nonergodic convergence rates have been established via a unified discrete Lyapunov function. Moreover, the semi-implicit method \cref{eq:apd-ex-x-im-l-x-h-g} has been applied to decentralized distributed optimization and performances better than the methods in \cite{li_convergence_2017,Shi2015}.

Our differential equation solver approach provides a systematically way to design new primal-dual methods for problem \cref{eq:min-f-X-Ax-b}, and the tool of Lyapunov function renders an effective way for convergence analysis. For future works, we will
pay attention to the solution existence and the exponential decay for our APD flow system \cref{eq:apd-sys-v} in general nonsmooth setting. Besides, convergence analysis under inexact computation shall be considered as well. 

At last, it is worth extending the current continuous model together its numerical discretizations to two block case \cref{eq:2b}. As discussed in \cref{rem:jdmm}, both the semi-implicit discretization \cref{eq:apd-im-x-ex-l-v} and the explicit one \cref{eq:apd-ex-x-ex-l-x-h-g} can be applied to the two block case \cref{eq:2b} and lead to parallel ADMM-type methods. However, to get the rate $O(1/k^2)$, they require strong convexity of $f$. Hence, it would also be our ongoing work for developing new accelerated primal-dual splitting methods that can handle partially strongly convex objectives.

\bibliographystyle{abbrv}

\end{document}